\documentclass[12pt,a4paper,oneside,final,reqno]{amsart}

\usepackage[T1]{fontenc}
\usepackage[utf8]{inputenc}
\usepackage[dvipsnames]{xcolor}
\usepackage[english]{babel}
\usepackage{mathtools, amsthm, amssymb, amscd} 
\usepackage{newtxtext, newtxmath} 
\usepackage[DIV=11]{typearea} 
\usepackage{hyperref} 
\usepackage[notref,notcite]{showkeys} 
\usepackage{todonotes}
\usepackage{enumitem} 
\usepackage{xypic} 
\usepackage{tikz}
\usetikzlibrary{cd, decorations.markings, arrows} 

\DeclareMathAlphabet{\mathcal}{OMS}{cmsy}{m}{n} 

\definecolor{DarkPurple}{rgb}{0.40,0.0,0.20}
\hypersetup{
    colorlinks =true,
    linkcolor= DarkPurple, 
    citecolor = NavyBlue 
} 

\usetikzlibrary{cd, decorations.markings, arrows, matrix, calc}

\newcommand{\spann}{\textrm{span}}
\newcommand{\coker}{\operatorname{coker}}

\newcommand{\G}{\mathcal{G}}

\newcommand{\ZZ}{\mathbb{Z}}
\newcommand{\NN}{\mathbb{N}}

\DeclareMathOperator{\id}{id}

\DeclareMathOperator{\im}{im}

\newtheorem{lemma}{Lemma}[section]
\newtheorem{corollary}[lemma]{Corollary}
\newtheorem{theorem}[lemma]{Theorem}
\newtheorem*{matuiAH*}{Matui's AH~Conjecture} 
\newtheorem{proposition}[lemma]{Proposition}

\newtheorem{assumption}[lemma]{Assumption}

\theoremstyle{definition}
\newtheorem{definition}[lemma]{Definition}
\newtheorem{example}[lemma]{Example}
\newtheorem{remark}[lemma]{Remark}

\title[Katsura--Exel--Pardo Groupoids and the AH~Conjecture]{Katsura--Exel--Pardo Groupoids and the AH~Conjecture}
 
\date{\today}

\author[1]{Petter Nyland}
\author[2]{Eduard Ortega}

\address{Department of Mathematical Sciences, Faculty of Information Technology and Electrical Engineering, NTNU -- Norwegian University of Science and Technology, Trondheim, Norway}
\email{petter.nyland@ntnu.no}
\address{Department of Mathematical Sciences, Faculty of Information Technology and Electrical Engineering, NTNU -- Norwegian University of Science and Technology, Trondheim, Norway}
\email{eduard.ortega@ntnu.no}


\numberwithin{equation}{section} 

\begin{document}

\begin{abstract}
It is proven that Matui's AH~conjecture is true for Katsura--Exel--Pardo groupoids $\mathcal{G}_{A,B}$ associated to integral matrices $A$ and $B$. This conjecture relates the topological full group of an ample groupoid with the homology groups of the groupoid. We also give a criterion under which the topological full group $\left\llbracket \mathcal{G}_{A,B} \right\rrbracket$ is finitely generated.
\end{abstract}

\maketitle

\section{Introduction} 

The \emph{AH~conjecture} is one of two conjectures formulated by Matui in~\cite{MatProd} concerning certain ample groupoids over Cantor spaces. This conjecture predicts that the abelianization of the topological full group of such a groupoid together with its first two homology groups fit together in an exact sequence as follows: \[\begin{tikzcd}
H_0(\mathcal{G}) \otimes \mathbb{Z}_2 \arrow{r}{j} & \llbracket \mathcal{G} \rrbracket_{\text{ab}} \arrow{r}{I_\text{ab}} & H_1(\mathcal{G}) \arrow{r} & 0.
\end{tikzcd}  \]

So far, the AH~conjecture has been confirmed in a number of cases. For instance, it holds for groupoids which 
 are both almost finite and principal~\cite{MatHom}. This includes AF-groupoids, transformation groupoids of higher-dimensional Cantor minimal systems and groupoids associated to aperiodic quasicrystals (as described in~\cite[Subsection~6.3]{Nek19}). At the opposite end of the spectrum, the AH~conjecture is also true for (products of) SFT-groupoids~\cite{MatProd}. The same goes for transformation groupoids associated to odometers~\cite{Scarp}, which incidentally, provided counterexamples to the other conjecture from~\cite{MatProd}, namely the \emph{HK~conjecture}. In the recent paper~\cite{NO20}, we showed that the AH~conjecture holds for graph groupoids of infinite graphs, complementing Matui's result in the finite case~\cite{MatTFG}.

The present paper may be viewed as a follow-up to~\cite{NO20}. Here we investigate the validity of the AH~conjecture for a class of groupoids known as \emph{Katsura--Exel--Pardo groupoids}. These groupoids are built from two equal-sized row-finite integer matrices~$A$ and~$B$, where $A$ has no negative entries, and are denoted $\mathcal{G}_{A,B}$. Their origins stem from Katsura's paper~\cite{Kat08}, in which he constructed $C^*$-algebras~$\mathcal{O}_{A,B}$---which we call \emph{Katsura algebras}---from such matrices.  Katsura showed that every Kirchberg algebra (in the UCT class) is stably isomorphic to some $\mathcal{O}_{A,B}$ and used this concrete realization to prove results pertaining lifts of actions on the $K$-groups of Kirchberg algebras. The Katsura algebras~$\mathcal{O}_{A,B}$ first appear as examples of topological graph $C^*$-algebras in~\cite{KatIV}.

Some years later later, Exel and Pardo introduced the notion of a \emph{self-similar graph}, and showed how to construct a $C^*$-algebra from this data, in~\cite{EP17}. This generalized Nekrashevych's construction from self-similar groups\footnote{A self-similar group may be viewed as a self-similar graph where the graph has only one vertex.} in~\cite{Nek09}. On the other hand, the construction of Exel and Pardo also encompassed the Katsura algebras. They realized that the matrices $A$ and $B$ could be used to describe a self-similar action by the integer group~$\ZZ$ on the graph whose adjacency matrix is~$A$ in such a way that the associated $C^*$-algebra becomes~$\mathcal{O}_{A,B}$. Exel and Pardo also gave a groupoid model for their $C^*$-algebras, and it is the groupoid associated with the aforementioned $\ZZ$-action that we call the \emph{Katsura--Exel--Pardo groupoid}. See Section~\ref{sec:KEP} for details.

The second author computed the homology groups of the Katsura--Exel--Pardo groupoids in~\cite{Ort} (under the assumption of pseudo-freeness, see Subsection~\ref{subsec:characterizations}), and found that the homology groups of $\mathcal{G}_{A,B}$ sum up to the $K$-theory of $C_r^* \left( \mathcal{G}_{A,B} \right) \cong \mathcal{O}_{A,B}$ in accordance with Matui's HK~conjecture~\cite[Conjecture~2.6]{MatProd}. 

In the present paper we make use of the description of the homology groups of $\mathcal{G}_{A,B}$ from~\cite{Ort} to show that the AH~conjecture holds whenever $\mathcal{G}_{A,B}$ is Hausdorff and effective and the matrix $A$ is finite and irreducible (Corollary~\ref{cor:AH}).

There are two subgroupoids of $\mathcal{G}_{A,B}$ that play important roles in the proof. One is the SFT-groupoid $\mathcal{G}_{A} \cong \mathcal{G}_{A,0}$ associated to the matrix $A$. The other is the kernel of the canonical cocycle on $\mathcal{G}_{A,B}$, denoted $\mathcal{H}_{A,B}$. Unlike the case of SFT-groupoids (or graph groupoids), the kernel of the cocycle is no longer an AF-groupoid. This means that we also need to take $H_1 \left( \mathcal{H}_{A,B} \right)$ into account when describing $H_1 \left( \mathcal{G}_{A,B} \right)$. A key observation that drives our proof is that the topological full group $\left\llbracket \mathcal{G}_{A,B} \right\rrbracket$ can be decomposed as $\left\llbracket \mathcal{G}_{A,B} \right\rrbracket = \llbracket \mathcal{H}_{A,B} \rrbracket  \llbracket \mathcal{G}_{A} \rrbracket$,  
when viewing $ \left\llbracket \mathcal{H}_{A,B} \right\rrbracket$  and $\left\llbracket \mathcal{G}_{A} \right\rrbracket$ as subgroups of $\left\llbracket \mathcal{G}_{A,B} \right\rrbracket$.  

We also investigate whether the topological full group $\left\llbracket \mathcal{G}_{A,B} \right\rrbracket$ is finitely generated. Matui has shown that topological full groups of (irreducible) SFT-groupoids are  finitely presented~\cite{MatTFG}. In the same vein, topological full groups associated to self-similar groups were shown to be finitely presented by Nekrashevych whenever the self-similar group is \emph{contracting}~\cite{Nek18}. We extend Nekrashevych's notion of a \emph{contracting} self-similar group to self-similar graphs and show that the self-similar graph associated to the pair of matrices~$A$ and~$B$ is contracting, assuming that  $B$ is entrywise smaller than $A$. Combining this with the finite generation of $\left\llbracket \mathcal{G}_{A} \right\rrbracket$, we show in Theorem~\ref{thm:finGen} that $\left\llbracket \mathcal{G}_{A,B} \right\rrbracket$ is then indeed finitely generated. In contrast, if~$E$ is a graph with an infinite emitter, then the topological full group~$\left\llbracket \mathcal{G}_E \right\rrbracket$ is not finitely generated~\cite[Proposition~10.1]{NO20}.    

We emphasize that the Katsura--Exel--Pardo groupoids are merely prominent special cases of the tight groupoids constructed from self-similar graphs in~\cite{EP17}. Moreover, this construction was further generalized to non-row-finite graphs in~\cite{EPS}. It is therefore a natural question whether the results of this paper can be generalized to other groupoids arising from self-similar graphs. A few things that make the Katsura--Exel--Pardo groupoids particularly nice to work with is that the self-similar action is explicitly given in terms of the matrices $A$ and $B$, the action does not move vertices, and the acting group is abelian (the ``most elementary'' abelian group even).  We believe that the methods employed in this paper could work well for other self-similar graphs where the acting group is abelian and the action fixes the vertices. 

This paper is organized as follows. In Section~\ref{sec:AH}, we briefly recall Matui's AH~conjecture and give references to the necessary preliminaries. The construction of the Katsura--Exel--Pardo groupoid is recalled in detail in Section~\ref{sec:KEP}. Then Hausdorffness, effectiveness and minimality of $\mathcal{G}_{A,B}$ is characterized in terms of the matrices $A$ and $B$. We also observe that if $\mathcal{G}_{A,B}$ satisfies the assumptions in the AH~conjecture, then $\mathcal{G}_{A,B}$ must be purely infinite. In Section~\ref{sec:hom}, we describe the first two homology groups of $\mathcal{G}_{A,B}$. This is done using a long exact sequence that relates the homology groups of $\mathcal{G}_{A,B}$ to those of the kernel groupoid~$\mathcal{H}_{A,B}$.
Our main result, namely that the AH~conjecture is true for Katsura--Exel--Pardo groupoids, is proved in Section~\ref{sec:TR}. Finally, in Section~\ref{sec:finGen}, we prove that $\left\llbracket \mathcal{G}_{A,B} \right\rrbracket$ is finitely generated, provided that $B$ is entrywise smaller than $A$.

\section{The AH~conjecture}\label{sec:AH}


%

As mentioned in the introduction, this paper is a follow-up to our recent paper~\cite{NO20}.
We treat the same problem---namely the AH~conjecture---for a related, but different, class of groupoids. Since the setting is so similar we have chosen to not give an extensive section covering preliminaries, but rather refer the reader to Section~2 of~\cite{NO20} 
and adapt all notation and conventions from there. Topics covered there include ample groupoids, topological full groups, homology of ample groupoids, cocycles and skew products. The reader is hereby warned that notation from~\cite[Section~2]{NO20}
henceforth will be used directly without reference.

Let us move on to describing the AH~conjecture, which predicts a precise relationship between the topological full group and the two first homology groups. For further details, consult~\cite[Section~4]{NO20}.   

\begin{matuiAH*}[{\cite[Conjecture~2.9]{MatProd}}]
Let~${\mathcal{G}}$ be an effective minimal second countable Hausdorff ample groupoid whose unit space $\mathcal{G}^{(0)}$ is a Cantor space. Then the following sequence is exact:
\begin{equation*}
\begin{tikzcd}
H_0(\mathcal{G}) \otimes \mathbb{Z}_2 \arrow{r}{j} & \left\llbracket \mathcal{G} \right\rrbracket_{\text{ab}} \arrow{r}{I_{\text{ab}}} & H_1(\mathcal{G}) \arrow{r} & 0.
\end{tikzcd} 
\end{equation*}
\end{matuiAH*}

The \emph{index map} $I \colon \left\llbracket \mathcal{G} \right\rrbracket \to H_1(\G)$ is the homomorphism given by $\pi_U\mapsto \left[1_{U} \right]$, where $U$ is a full bisection in $\mathcal{G}$, and the induced map on the abelianization~$\llbracket \mathcal{G} \rrbracket_{\text{ab}}$ is denoted $I_{\text{ab}}$. The map $j$ will not be used directly (see e.g.~\cite[Subsection~4.1]{NO20} for its definition).   

Recall the notion of transpositions in the topological full group from~\cite[Subsection~2.2]{NO20}.   
We will let $\mathcal{T}(\mathcal{G})$ denote the subgroup of~$\left\llbracket \mathcal{G} \right\rrbracket$ generated by all transpositions.\footnote{In~\cite{NO20},
 the subgroup generated by all transpositions is denoted $\mathcal{S}(\mathcal{G})$, but for $\mathcal{G} = \mathcal{G}_{A,B}$ we find this to be too similar to the set $\mathcal{S}_{A,B}$ that is defined in Subsection~\ref{subsec:tightgrpd}.} One always has $\mathcal{T}(\mathcal{G}) \subseteq \ker(I)$, and having equality is closely related to the AH~conjecture. 

\begin{definition}[{\cite[Definition~2.11]{MatProd}}]
Let $\mathcal{G}$ be an effective ample Hausdorff groupoid. We say that~$\mathcal{G}$ has \emph{Property~TR} if $\mathcal{T}(\mathcal{G}) = \ker(I)$.
\end{definition}

In the next section, we will see that the Katsura--Exel--Pardo groupoids are purely infinite (in the sense of~\cite[Definition~4.9]{MatTFG}). It then follows that the AH~conjecture is equivalent to having Property~TR for these (see~\cite[Remark~4.12]{NO20}). 
The main goal therefore becomes to establish Property~TR for $\mathcal{G}_{A,B}$.

\section{The Katsura--Exel--Pardo groupoid}\label{sec:KEP}

In this section we recall the construction of the The Katsura--Exel--Pardo groupoid $\mathcal{G}_{A,B}$ from~\cite{EP17},  
and we recall some of its properties.

\subsection{The self similar action by $\ZZ$ on the graph $E_A$}

Let us begin explaining the construction. Let $N \in \NN \cup \{\infty\}$ and let $A$ and $B$ be two row-finite $N \times N$ integral matrices. We require that all entries in $A$ are non-negative and that $A$ has no zero-rows. For the construction we may also assume without loss of generality that $B_{i,j} = 0$ whenever $A_{i,j} = 0$. Let~$E_A$ denote the (directed) graph whose adjacency matrix is $A$. For graphs we freely adopt notation and conventions from~\cite[Section~3]{NO20}.
In addition to that, given a finite path~$\mu = e_1 e_2 \cdots e_k \in E_A^*$ and an index $1 \leq j \leq k$, the subpath $e_1 e_2 \ldots e_j$ is denoted $\mu \vert_{j}$. We will call a matrix \emph{essential} if it has no zero-rows and no zero-columns. 

We will now describe how the matrices $A$ and $B$ give rise to a self-similar action by the integer group~$\ZZ$ on the graph $E_A$ as in the framework of~\cite{EP17}. In the next subsection, we will describe the associated (tight) groupoid. 

\begin{remark}
We remark that Exel and Pardo use the opposite convention for paths in~\cite{EP17}, which means that their paths go ``backwards'' in the graph. 
\end{remark}

To describe the action $\kappa \colon \ZZ \curvearrowright E_A$ we need to fix an (arbitrary) enumeration of the edges in $E_A$ as follows 
\begin{equation*}
E_A^1 = \{e_{i,j,n} \mid 1 \leq i,j \leq N, \ 0 \leq n < A_{i,j} \}. 
\end{equation*}
Then  $s ( e_{i,j,n} ) = i$ and $r ( e_{i,j,n} ) = j$, when enumerating the vertices as ${E_A^0 = \{1,2, \ldots, N\}}$.
Let $m \in \ZZ$ and $e_{i,j,n} \in E_A^1$ be given. By the division algorithm there are unique integers~$q$ and~$r$ satisfying 
\begin{equation*}
m B_{i,j} + n = q A_{i,j} + r \quad  \text{and} \quad 0 \leq r < A_{i,j}.   
\end{equation*}
The action $\kappa$ is defined to be trivial on the vertices (i.e.~$\kappa_m(i) = i$), and on edges it is given by
\begin{equation*}
\kappa_m(e_{i,j,n}) \coloneqq e_{i,j,r}. 
\end{equation*}
In words $\kappa_m$ maps the $n$'th edge between the vertices $i$ and $j$ to the $r$'th edge, where $r$ is the remainder of $m B_{i,j} + n$ modulo $A_{i,j}$. The associated \emph{one-cocycle} $\varphi \colon \ZZ \times E_A^1 \to \ZZ$ is given by
\begin{equation*}
\varphi(m, e_{i,j,n}) \coloneqq q.
\end{equation*}
The cocycle condition \[\varphi(m_1 + m_2, e) = \varphi(m_1, \kappa_{m_2}(e)) + \varphi(m_2,e)\] is easily seen to be satisfied. That same computation shows that $\kappa_{m_1 + m_2} = \kappa_{m_1} \circ \kappa_{m_2}$. Furthermore, the standing assumption (2.3.1) on page 1051 of~\cite{EP17} is trivially satisfied since $\kappa$ fixes the vertices. Note that $\varphi(0, e) = 0$ and $\kappa_0(e) = e$ for all~${e \in E_A^1}$.


As in~\cite[Proposition~2.4]{EP17} $\kappa$ and $\varphi$ extends inductively to finite paths by setting 
\begin{equation*}
\kappa_m(\mu e) \coloneqq \kappa_m(\mu) \kappa_{\varphi(m, \mu)}(e) \quad \text{and} \quad \varphi(m, \mu e) \coloneqq \varphi (\varphi(m, \mu), e)
\end{equation*}
for $\mu \in E_A^*$ and $e \in r(\mu) E_A^1$. Explicitly, for a finite path $\mu = e_1 e_2 \cdots e_k \in E_A^*$ we  have 
\begin{equation}\label{eq:kappafinite}
\kappa_m(\mu) = \kappa_m(e_1)        \kappa_{\varphi(m,e_1)}(e_2)       \kappa_{\varphi(m,e_1 e_2)}(e_3) \cdots \kappa_{\varphi(m,\mu \vert_{k-1})}(e_k)
\end{equation}
and 
\begin{equation}\label{eq:iteratephi}
\varphi(m, \mu) = \varphi\left(\varphi(\ldots (\varphi (\varphi(m, e_1), e_2), \ldots), e_{k-1}), e_k\right).
\end{equation}
By allowing Equation~\eqref{eq:kappafinite} to go on ad infinitum, $\kappa$ extends to an action on the infinite path space $E_A^\infty$. Note that we still have
\[\varphi(m_1 + m_2, \mu) = \varphi(m_1, \kappa_{m_2}(\mu)) + \varphi(m_2,\mu)\]
and 
\[\kappa_m(\mu \nu) = \kappa_m(\mu) \kappa_{\varphi(m, \mu)}(\nu)\]
for $\mu, \nu \in E_A^*$ with $r(\mu) = s(\nu)$. The latter formula also holds if $\nu$ is replaced by an infinite path.


\subsection{Describing the tight groupoid}\label{subsec:tightgrpd}

Define the set \[\mathcal{S}_{A,B} \coloneqq \{(\mu, m, \nu) \in E_A^* \times \ZZ \times E_A^* \mid r(\mu) = r(\nu)  \}.\]
In~\cite{EP17}, the set $\mathcal{S}_{A,B}$ is given the structure of an inverse $*$-semigroup which acts (partially) on the infinite path space $E_A^\infty$. In brief terms this partial action is given by 
\begin{equation*}
 (\mu, m, \nu) \cdot \nu y = \mu \kappa_m(y) \text{ for } y \in r(\nu) E_A^\infty.
\end{equation*}
Following~\cite{Ort} we skip directly to the concrete description of the tight groupoid $\mathcal{G}_\text{tight}(\mathcal{S}_{A,B})$ given in~\cite[Section~8]{EP17}. 

Consider the set of all quadruples $(\mu, m, \nu; x)$ where $(\mu, m, \nu) \in \mathcal{S}_{A,B}$ and $x \in Z(\nu)$. Then we can write $x = \nu e z$ for some $e \in E_A^1$ and $z \in E_A^\infty$. Let $\sim$ be the equivalence relation on this set of quadruples generated by the basic relation 
\begin{equation}\label{eq:eqrel}
(\mu, m, \nu; x) \sim (\mu \kappa_m(e), \varphi(m,e), \nu e; x).
\end{equation}
Denote the equivalence class of $(\mu, m, \nu; x)$ under $\sim$ by $\left[ \mu, m, \nu; x \right]$. In particular, we have 
\[\left[ \mu, m, \nu; x \right] = \left[ \mu \kappa_m(y \vert_j), \varphi(m,y \vert_j), \nu y \vert_j; x \right]\]
for each $j \in \mathbb{N}$, where $y$ is the infinite path satisfying $x = \nu y$. It is somewhat cumbersome to explicitly write this equivalence relation out, but it can be done as follows. Let~${(\mu, m, \nu), (\lambda, n, \tau) \in \mathcal{S}_{A,B}}$, $x \in Z(\nu)$ and $z \in Z(\tau)$. Then \[\left[ \mu, m, \nu; x \right] = \left[ \lambda, n, \tau; z \right] \] if and only if
\begin{itemize}
\item $x = z$, so then $x = \nu y = \tau w$ for some infinite paths $y$ and $w$. In particular, $\nu$ is a subpath of $\tau$ or vice versa. 
\item $\left| \mu \right| - \left|  \nu \right| = \left| \lambda \right| - \left|  \tau \right|$.
\item $\mu \kappa_m(y) = \lambda \kappa_n(w)$.
\item $\varphi(m, y \vert_j) = \varphi(n, w \vert_l)$ for some $j,l \in \mathbb{N}$ with $l - j = \left| \mu \right| - \left|  \nu \right|$.
\end{itemize}

We define the \emph{Katsura--Exel--Pardo groupoid} to be 
\[\mathcal{G}_{A,B} \coloneqq \left\{\left[ \mu, m, \nu; x \right] \mid (\mu, m, \nu) \in \mathcal{S}_{A,B}, \ x \in Z(\nu) \right\}. \]
Writing $x = \nu y$, the inverse operation is given by 
\[\left[ \mu, m, \nu; x \right]^{-1} \coloneqq \left[ \nu, -m, \mu; \mu \kappa_m(y) \right].   \] 
The composable pairs are 
\[\mathcal{G}_{A,B}^{(2)} \coloneqq \left\{ \left(  \left[ \lambda, n, \tau; z \right]   ,   \left[ \mu, m, \nu; \nu y \right] \right)  \in \mathcal{G}_{A,B} \times \mathcal{G}_{A,B} \mid  \mu \kappa_m(y) = z  \right\} \]
and the product is given by 
\[  \left[ \lambda, n, \tau; z \right]   \cdot  \left[ \mu, m, \nu; x \right]  \coloneqq  \left[ \lambda \kappa_m(\tau'), \varphi(n, \tau') +m, \nu; x \right],     \]
in the case that $\mu = \tau \tau'$. In the case that $\tau = \mu \mu'$ the formula is slightly more complicated, so let us instead use the equivalence relation $\sim$ to state a simpler ``standard form'' for the product. Using the basic relation~\eqref{eq:eqrel} we can choose representatives with $\left| \tau \right| = \left| \mu \right|$, which forces $\tau = \mu$. Hence every composable pair and their product can be represented as 
\[   \left[ \lambda, n, \mu; \mu \kappa_m(y) \right]   \cdot  \left[ \mu, m, \nu; \nu y \right] = \left[ \lambda, n+m, \nu; \nu y \right].  \]

The source and range maps are given by 
\begin{align*}
s\left(\left[ \mu, m, \nu; \nu y \right]  \right) &=  \left[ \nu, 0, \nu; \nu y \right] = \left[ s(\nu), 0, s(\nu); \nu y \right],  \\
r\left(\left[ \mu, m, \nu; \nu y \right]  \right) &=  \left[ \mu, 0, \mu; \mu \kappa_m(y) \right] = \left[ s(\mu), 0, s(\mu); \mu \kappa_m(y) \right]. 
\end{align*}
Thus we may identify the unit space $\mathcal{G}_{A,B}^{(0)}$ with the infinite path space $E_A^\infty$ under the correspondence $\left[ s(x), 0, s(x); x \right] \leftrightarrow x$. This correspondence is also compatible with the topology on $\mathcal{G}_{A,B}$ that will be specified shortly. The source and range maps become 
\[s\left(\left[ \mu, m, \nu; x \right]  \right) = x \quad \text{and} \quad r\left(\left[ \mu, m, \nu; \nu y \right]  \right) = \mu \kappa_m(y).  \]

For a triple $(\mu, m, \nu) \in \mathcal{S}_{A,B}$ we define
\[ Z(\mu, m, \nu) \coloneqq \left\{ \left[ \mu, m, \nu; x \right] \mid x \in Z(\nu)   \right\}.\]
These sets form a basis for the topology on $\mathcal{G}_{A,B}$, in which each basic set $Z(\mu, m, \nu)$ is a compact open bisection~\cite[Proposition~9.4]{EP17}. Note that 
\[ s\left(Z(\mu, m, \nu) \right) = Z(\nu) \quad \text{and} \quad  r\left(Z(\mu, m, \nu) \right) = Z(\mu).  \]
The Katsura--Exel--Pardo groupoid $\mathcal{G}_{A,B} \cong \mathcal{G}_\text{tight}(\mathcal{S}_{A,B})$ is ample, second countable and amenable~\cite{EP17}. However, it is not always Hausdorff. This, and other properties, will be characterized in the next subsection. 

An important observation that will be exploited in several of the coming proofs is that the graph groupoid $\mathcal{G}_{E_A}$ is isomorphic to $\mathcal{G}_{A,0}$, and moreover embeds canonically into $\mathcal{G}_{A,B}$ for any matrix $B$. Observe that in $\mathcal{G}_{A,0}$ we have 
$\left[ \mu, m, \nu; \nu y \right] = \left[ \mu, 0, \nu; \nu y \right]$ for each $m \in \ZZ$. Hence mapping $\left[ \mu, 0, \nu; \nu y \right]$ to $(\mu y, \left| \mu \right| - \left| \nu \right|, \nu y )$ yields an isomorphism between $\mathcal{G}_{A,0}$ and~$\mathcal{G}_{E_A}$. Furthermore, it is clear that $\left[ \mu, 0, \nu; x \right] \mapsto \left[ \mu, 0, \nu; x \right]$ gives an étale embedding~$\mathcal{G}_{A,0} \hookrightarrow \mathcal{G}_{A,B}$ which preserves the unit space. 

Another special case is when $A = B$. Then we have $\mathcal{G}_{A,A} \cong \mathcal{G}_{A} \times \ZZ$ (where~$\ZZ$ is viewed as a group(oid)). These groupoids fall outside of the scope of the AH~conjecture, however, for they are far from being effective.



\subsection{When is $\mathcal{G}_{A,B}$ Hausdorff, effective and minimal?}\label{subsec:characterizations}

We begin by noting that $\mathcal{G}_{A,B}$ has compact unit space if and only if $N < \infty$ (i.e.\ $A$ and $B$ are finite matrices). And in this case it is a Cantor space precisely when $E_A$ satisfies Condition~(L).

Before characterizing Hausdorfness precisely, we discuss a sufficient condition known as \emph{pseudo-freeness}. This is an underlying assumption in~\cite{Ort}. The action ${\kappa \colon \ZZ \curvearrowright E_A}$ is called \emph{pseudo-free} if $\kappa_m(e) = e$ and $\varphi(m, e) = 0$ implies $m = 0$, for $m \in \ZZ$ and~${e \in E_A^1}$ (see~\cite[Definition~5.4]{EP17} for the general definition). Combining Lemma~18.5 and Proposition~12.1 from~\cite{EP17} yields the following.   

\begin{proposition}[{\cite{EP17}}] 
The action $\kappa \colon \ZZ \curvearrowright E_A$ is pseudo-free if and only if $A_{i,j} = 0$ whenever $B_{i,j} = 0$. And when this is the case $\mathcal{G}_{A,B}$ is Hausdorff.
\end{proposition}

A precise characterization of when $\mathcal{G}_{A,B}$ is Hausdorff is the following.

\begin{proposition}[{\cite[Theorem~18.6]{EP17}}] 
The following are equivalent:
\begin{enumerate}
\item The Katsura--Exel--Pardo groupoid $\mathcal{G}_{A,B}$ is Hausdorff.
\item Whenever $B_{i,j} = 0$ while $A_{i,j} \geq 1$, then for any $m \in \ZZ \setminus \{0\}$ the set \[ \left\{\mu \in E_A^* \mid r(\mu) = i \ \& \   m \frac{B_{\mu \vert_t}}{A_{\mu \vert_t}} \in \ZZ \setminus \{0\} \text{ for } 1 \leq t \leq \left| \mu \right| \right\} \] is finite.
\end{enumerate}
\end{proposition}

\begin{remark}
There is a small misprint in the statement of~\cite[Theorem~18.6]{EP17}, which is why the statement above differs slightly (even after reversing the direction of the edges).
\end{remark}

%
%
%


The minimality of $\mathcal{G}_{A,B}$ turns out to be independent of the matrix $B$, and is only governed by the minimality of the graph groupoid $\mathcal{G}_{E_A}$.

\begin{proposition}[{\cite[Theorem~18.7]{EP17}}]\label{prop:minimal}
The Katsura--Exel--Pardo groupoid $\mathcal{G}_{A,B}$ is minimal if and only if the graph $E_A$ is cofinal. 
\end{proposition}

In particular, if the matrix $A$ is irreducible (which is equivalent to $E_A$ being strongly connected), then $\mathcal{G}_{A,B}$ is minimal. The converse holds if $E_A$ has no sources (nor sinks).

\begin{remark}
Proposition~\ref{prop:minimal} actually holds for any self-similar graph in which the vertices are fixed. A general  characterization is given in~\cite[Theorem~13.6]{EP17}.
\end{remark}

Let us move on to characterizing when $\mathcal{G}_{A,B}$ is effective.\footnote{The term ``essentially principal'' is used in~\cite{EP17}.}

\begin{proposition}[{\cite[Theorem~18.8]{EP17}}] 
The following are equivalent:
\begin{enumerate}
\item The Katsura--Exel--Pardo groupoid $\mathcal{G}_{A,B}$ is effective.
\item \begin{enumerate}
\item The graph $E_A$ satisfies Condition~(L).
\item If $1 \leq i \leq N$, $m \in \ZZ \setminus \{0\}$, and for all $x \in Z(i)$ we have $\displaystyle m \frac{B_{x \vert_t}}{A_{x \vert_t}} \in \ZZ$ for all $t \in \NN$, then there exists $T \in \NN$ such that $B_{x \vert_T} = 0$ for all $x \in Z(i)$.
\end{enumerate}
\end{enumerate}
\end{proposition}

The premise in $(2b)$ above is fairly strong, as it stipulates that $\kappa_m(x) = x$ for all $x \in Z(i)$. In many cases this will not happen for any vertex $i$, which means $(2b)$ is trivially satisfied. One such case is the following.  

\begin{corollary}[{\cite[Corollary~18.9]{EP17}}]\label{cor:sufficientEffective} 
If $E_A$ satisfies Condition~(L) and for each $1 \leq i \leq N$ there exists $x \in Z(i)$ such that $B_{x \vert_t} \neq 0$ for all $t \in \mathbb{N}$ and  $\displaystyle \lim_{t\to \infty} \frac{B_{x \vert_t}}{A_{x \vert_t}} = 0$, then $\mathcal{G}_{A,B}$ is effective.
\end{corollary}

The following is a class of examples to which Corollary~\ref{cor:sufficientEffective} applies.

\begin{example}
If the matrices $A, B$ satisfy  $A_{i,i} \geq 2$ and $0 < \left| B_{i,i} \right| < A_{i,i}$ for all $1 \leq i \leq N$, then $\mathcal{G}_{A,B}$  is effective. If $A$ is irreducible it suffices that this condition holds for a single vertex $i$.
\end{example}

The following remark illustrates that the class of examples above is already fairly rich.

\begin{remark}
It suffices to consider matrices $A, B$ satisfying $A_{i,i} \geq 2$ and $B_{i,i} = 1$ for each $1 \leq i \leq N$ with $A$ irreducible for $\mathcal{O}_{A,B}$ to exhaust all Kirchberg algebras up to stable isomorphism~\cite[Proposition~4.5]{KatIV}. 
\end{remark}




Next we observe that the Katsura--Exel--Pardo groupoids that satisfy the assumptions of the AH~conjecture are purely infinite (in the sense of~\cite[Definition~4.9]{MatTFG}). This means that the index map is surjective~\cite[Theorem~5.2]{MatTFG}, so we only need to establish Property~TR in order to prove that the AH~conjecture hold for these groupoids.

\begin{proposition}\label{prop:PI}
Let $N < \infty$ and assume that $\mathcal{G}_{A,B}$ is Hausdorff, effective and minimal. Then $\mathcal{G}_{A,B}$ is purely infinite.
\end{proposition}
\begin{proof}
Since the SFT-groupoid $\mathcal{G}_A \cong \mathcal{G}_{A,0}$ is an open ample subgroupoid of $\mathcal{G}_{A,B}$, the pure infiniteness of $\mathcal{G}_{A,B}$ follow from that of $\mathcal{G}_A$, which is established in~\cite[Lemma~6.1]{MatTFG}.
\end{proof}

As in~\cite{NO20} 
we make the following ad hoc definition for brevity. 

\begin{definition}
We say that the matrices $A, B$ satisfy the \emph{AH~criteria} if $N < \infty$ and $\mathcal{G}_{A,B}$ is Hausdorff, effective and minimal. 
\end{definition}

A large class of pairs of matrices satisfying the AH~criteria are given in the following example.

\begin{example}
Let $N \in \mathbb{N}$ and let $A \in M_N(\ZZ_+)$, $B \in M_N(\ZZ)$. Assume that $A$ is irreducible and that $B_{i,j} = 0$ if and only if $A_{i,j} = 0$. Assume further that there exists some $i$ between $1$ and $N$ such that $\left| B_{i,i} \right| < A_{i,i} \geq 2$. Then the matrices $A, B$ satisfy the AH~criteria.
\end{example}


\section{The homology of $\mathcal{G}_{A,B}$}\label{sec:hom}

In this section, we will describe the homology groups of the Katsura--Exel--Pardo groupoids, following~\cite{Ort}. Although the action is assumed to be pseudo-free throughout in~\cite{Ort}, most of what we need here also work without this assumption, with one notable exception which is adressed in Equation~\eqref{eq:Homn} below.

\begin{assumption}
We assume throughout that $N < \infty$ and that $\mathcal{G}_{A,B}$ is Hausdorff. 
\end{assumption}
%




\subsection{The kernel subgroupoid $\mathcal{H}_{A,B}$}
Similarly to the canonical cocycle on an SFT-groupoid (see page 37 of~\cite{MatHom}) we can define a cocycle\footnote{I.e.\ a continuous groupoid homomorphism into a group.} $c \colon \mathcal{G}_{A,B} \to \ZZ$ on a Katsura--Exel--Pardo groupoid by setting \[c\left( [\mu, m, \nu;x]  \right) = \left| \mu \right| - \left|  \nu \right|. \]
This is well-defined since the difference $\left| \mu \right| - \left|  \nu \right|$ is preserved under $\sim$. Now define 
\[ \mathcal{H}_{A,B} \coloneqq \ker(c) = \left\{ [\mu, m, \nu;x] \in \mathcal{G}_{A,B} \mid \left| \mu \right| = \left|  \nu \right| \right\},  \]
which is a clopen ample subgroupoid of $\mathcal{G}_{A,B}$. In contrast to the case of graph groupoids, this kernel is generally not an AF-groupoid (it need not be principal), but it is still key to computing the homology of $\mathcal{G}_{A,B}$. 

Next, for each $n \in \mathbb{N}$ we define the open subgroupoid
\[\mathcal{H}_{A,B,n} \coloneqq \left\{ [\mu, m, \nu;x] \in \mathcal{G}_{A,B} \mid \left| \mu \right| = \left|  \nu \right| = n \right\} \subseteq \mathcal{H}_{A,B}.  \] 
Observe that $\mathcal{H}_{A,B,n}  \subseteq \mathcal{H}_{A,B,n+1} $ by~\eqref{eq:eqrel} and that $\cup_{n=1}^\infty \mathcal{H}_{A,B,n} = \mathcal{H}_{A,B}$. Hence 
\begin{equation}\label{eq:limitH}
H_i \left(\mathcal{H}_{A,B} \right) \cong \varinjlim\left(H_i \left(\mathcal{H}_{A,B,n} \right), H_i(\iota_n) \right)
\end{equation}
by~\cite[Proposition~4.7]{FKPS}, where $\iota_n$ is the inclusion map.

It follows from the proof of~\cite[Proposition~2.3]{Ort} that if $\mu,\nu \in E_A^n$ and $r(\mu) = r(\nu)$, then
\begin{equation*}
\left[1_{Z(\mu)} \right] = \left[1_{Z(\nu)} \right] \in H_0\left(\mathcal{H}_{A,B,n} \right)
\end{equation*}
and that we have
\begin{equation}\label{eq:H0n}
H_0\left(\mathcal{H}_{A,B,n} \right) = \spann\left\{\left[1_{Z(\mu)} \right] \mid \mu \in E_A^n  \right\} \cong \ZZ^N,
\end{equation}
even without the assumption of pseudo-freeness. The isomorphism in~\eqref{eq:H0n} is given by mapping $\left[1_{Z(\mu)} \right]$  to $1_{r(\mu)}$, where we by $1_w \in \ZZ^N \cong \oplus_{v \in E_A^0} \ZZ$ for $w \in E_A^0$ mean the tuple with~$1$ in the $w$'th coordinate and $0$ elsewhere.

As for $H_1\left(\mathcal{H}_{A,B,n} \right)$, for paths $\mu$ and $\nu$ as above it similarly follows from the proof of~\cite[Proposition~2.4]{Ort} that 
\begin{align}\label{eq:indicatorHomology}
\left[1_{Z(\mu,m, \mu)} \right] &= \left[1_{Z(\mu,m, \nu)} \right] = \left[1_{Z(\nu,m, \nu)} \right] \in H_1\left(\mathcal{H}_{A,B,n} \right), \nonumber \\
\left[1_{Z(\mu,m, \mu)} \right] &= m \left[1_{Z(\mu,1, \mu)} \right] \in H_1\left(\mathcal{H}_{A,B,n} \right),
\end{align}
and hence 
\begin{equation}\label{eq:H1span}
H_1\left(\mathcal{H}_{A,B,n} \right) = \spann\left\{\left[1_{Z(\mu,1, \mu)} \right] \mid \mu \in E_A^n  \right\}.
\end{equation}
If the action is pseudo-free, then
\[H_1\left(\mathcal{H}_{A,B,n} \right) \cong \ZZ^N\]
by identifying $\left[1_{Z(\mu,1, \mu)} \right]$ with $1_{r(\mu)}$.

However, when the action is not pseudo-free, we need to take care. The group $H_1\left(\mathcal{H}_{A,B,n} \right)$ will still be a free abelian group, but its rank may be smaller than $N$. To explain this phenomenon, let us call a vertex $i \in \left\{ 1, 2, \ldots, N \right\}$  a \emph{$B$-sink} if $B_{i,j} = 0$ for all $j$ with $A_{i,j} > 0$. Any path passing through a $B$-sink will be strongly fixed\footnote{Meaning that $\kappa_m(\mu) = \mu$ and $\varphi(m, \mu) = 0$ (\cite[Definition~5.2]{EP17}).} by the action. To see the impact this has on $H_1\left(\mathcal{H}_{A,B,n} \right)$, suppose that $i$ is a $B$-sink and that $\mu \in E_A^n$ has $r(\mu) = i$. Then we have the counter-intuitive equality
\begin{equation}\label{eq:degenerateBisection}
Z(\mu, 1, \mu) = Z(\mu, 0, \mu) \subseteq \mathcal{G}_{A,B}^{(0)},
\end{equation}
since for any $x = \mu e z \in Z(\mu)$ with $e \in r(\mu) E_A^1$ we have 
\[ (\mu, 1, \mu; x) \sim (\mu \kappa_1(e), \varphi(1,e), \mu e; x) = (\mu e, 0, \mu e; x) \sim (\mu, 0, \mu; x).   \] 
This in turn means that $\left[1_{Z(\mu,1, \mu)} \right]  = 0 \in H_1\left(\mathcal{H}_{A,B,n} \right)$, so this part of $H_1\left(\mathcal{H}_{A,B,n} \right)$ collapses. More generally, the same will happen to any path $\mu \in E_A^n$ for which every infinite path~${x \in Z(r(\mu))}$ passes through a $B$-sink. To have a name for vertices for which this does not happen, let us define a vertex $1 \leq i \leq N$ to be a \emph{$B$-regular} if there exists a path $\mu$, containing no $B$-sinks, starting at $i$ which connects to a cycle that contain no $B$-sinks. This is the same as saying that there is some infinite path starting at $i$ which does not pass through any $B$-sink. Bisections $Z(\mu,1, \mu)$ with $r(\mu)$ $B$-regular behave just like in the pseudo-free case, while those with $r(\mu)$ not $B$-regular vanish in $H_1\left(\mathcal{H}_{A,B,n} \right)$ as explained above.  Let $R_B$ denote the number of $B$-regular vertices. Then we have that 
\begin{equation}\label{eq:Homn}
H_1\left(\mathcal{H}_{A,B,n} \right) = \spann\left\{\left[1_{Z(\mu,1, \mu)} \right] \mid \mu \in E_A^n \text{ with } r(\mu) \ B\text{-regular}  \right\} \cong \ZZ^{R_B}. 
\end{equation}
This particular description (as opposed to~\eqref{eq:H1span}) is only used in the proof of Lemma~\ref{lem:HnTR}.

\begin{remark}
By viewing the matrices $A$ and $B$ as endomorphisms of $\ZZ^N$ (via left multiplication) we may consider the inductive limits
\begin{equation}\label{eq:limitA}
\ZZ_A \coloneqq \varinjlim\left(\ZZ^N, A \right) \quad \text{and} \quad  \ZZ_B \coloneqq \varinjlim\left(\ZZ^N, B \right).  
\end{equation}
Let $\phi_{n, \infty}^A  \colon Z^N \to Z_A$ and $\phi_{n, \infty}^B  \colon Z^N \to Z_B$ denote the canonical maps into the inductive limits. Propositions~2.3~\&~2.4 in~\cite{Ort} remain valid without pseudo-freeness and they show that the inductive limits in~\eqref{eq:limitH} for $i = 0$ and $i = 1$ turn into the limits in~\eqref{eq:limitA}, respectively. This means that 
\[H_0\left(\mathcal{H}_{A,B} \right) \cong \ZZ_A \quad  \text{and} \quad H_1\left(\mathcal{H}_{A,B} \right) \cong \ZZ_B,  \]
where the isomorphisms are given by 
\[\left[1_{Z(\mu)} \right] \mapsto \phi_{n, \infty}^A \left(1_{r(\mu)} \right)  \quad \text{and} \quad \left[1_{Z(\mu,1, \mu)} \right] \mapsto \phi_{n, \infty}^B \left(1_{r(\mu)} \right),  \]
respectively, for $\mu \in E_A^n$. This is still compatible with Equation~\eqref{eq:Homn}, because if~$v$ is a non-$B$-regular vertex, then~$1_v$ is eventually annihilated in the inductive limit~$\ZZ_B$. What does not necessarily hold, however, without pseudo-freeness, is~\cite[Proposition~2.2]{Ort} which says that $H_n \left( \mathcal{H}_{A,B} \right) = 0$ when $n \geq 2$. But this part is not needed for the results in the present paper.
\end{remark}

Let $\mathcal{G}_{A,B} \times_c \ZZ$ denote the skew product associated to the cocycle $c$ (see~\cite[Subsection~2.5]{NO20}). 

\begin{lemma}\label{lem:full}
The clopen set $E_A^\infty \times \{0\} \subseteq  \left( \mathcal{G}_{A,B} \times_c \ZZ \right)^{(0)}$ is $\left(\mathcal{G}_{A,B} \times_c \ZZ \right)$-full.
\end{lemma}
\begin{proof}
The same proof as for SFT-groupoids works here (see~\cite[Lemma~6.1]{FKPS} for a more general result). Let $(x, k) \in E_A^\infty \times \ZZ = \left( \mathcal{G}_{A,B} \times_c \ZZ \right)^{(0)}$ be given. 
 If $k < 0$, then the groupoid element $\left[x \vert_{-k}, 0, r\left(x \vert_{-k} \right); x_{[-k+1, \infty]}\right] \in \mathcal{G}_{A,B} \times_c \ZZ$ has range $(x, k)$ and source~$(x_{[-k+1, \infty]}, 0)$, which shows that $E_A^\infty \times \{0\}$ meets the $\left(\mathcal{G}_{A,B} \times_c \ZZ \right)$-orbit of $(x, k)$. In the case that $k > 0$ we can, since $E_A$ is a finite graph without sinks, find an index~${n \in \mathbb{N}}$ for which~$r(x_n)$ supports a cycle. By concatenating along this cycle we can find a path~${\nu \in E_A^*}$ with $r(\nu) = r(x_n)$ and $\left| \nu \right| = n + k$. And then the element $\left[x \vert_{n}, 0, \nu; \nu x_{[n+1, \infty]}\right]$ has range~$(x, k)$ and source $(\nu x_{[n+1, \infty]}, 0) \in E_A^\infty \times \{0\}$.
\end{proof}

Recall that
\[\mathcal{H}_{A,B} \cong \left(\mathcal{G}_{A,B} \times_c \ZZ \right) \vert_{E_A^\infty \times \{0\}}\]
via the map
\[ [\mu, m, \nu;x] \mapsto \left([\mu, m, \nu;x], 0 \right).   \] 
Composing this with the inclusion of the restriction we obtain an embedding~$\iota$ of~$\mathcal{H}_{A,B}$ into the skew product $\mathcal{G}_{A,B} \times_c \ZZ$. Lemma~\ref{lem:full} says that $\mathcal{H}_{A,B}$ is Kakutani equivalent to $\mathcal{G}_{A,B} \times_c \ZZ$ from which we have the following consequence (by~\cite[Lemma~4.3]{FKPS}).

\begin{proposition}\label{prop:Kakeq}
The embedding $\iota \colon \mathcal{H}_{A,B} \to \mathcal{G}_{A,B} \times_c \ZZ$ induces isomorphisms \[H_i\left( \mathcal{H}_{A,B} \right) \cong H_i\left( \mathcal{G}_{A,B} \times_c \ZZ \right)\] for each $i \geq 0$.
\end{proposition}

\subsection{A long exact sequence in homology}

From~\cite[Proposition~6.1]{NO20} 
applied to the cocycle $c \colon \mathcal{G}_{A,B} \to \ZZ$ we obtain the following long exact sequence in homology: 
\begin{equation}\label{eq:LES}
\begin{tikzcd}[column sep = 2.7em]
\cdots \arrow{r} &[-20pt]    H_2(\mathcal{G}_{A,B}) \arrow{r}{\partial_2} &[-7pt] H_1(\G_{A,B} \times_c \ZZ) \arrow{r}{\color{red}  \id - H_1(\rho_\bullet)} &   H_1(\G_{A,B} \times_c \ZZ) \arrow{r}{\color{Aquamarine}H_1(\pi_\bullet)} & \textcolor{Aquamarine}{(a)}  \\ 
\textcolor{Aquamarine}{(a)} \arrow{r} &  H_1(\mathcal{G}_{A,B})    \arrow{r}{\color{blue} \partial_1} & H_0(\G_{A,B} \times_c \ZZ) \arrow{r}{\color{Sepia} \id - H_0(\rho_\bullet)} &   H_0(\G_{A,B} \times_c \ZZ) \arrow{r}{H_0(\pi_\bullet)} & H_0(\mathcal{G}_{A,B})  \arrow{r} &[-20pt] 0.
\end{tikzcd}
\end{equation}
Consult~\cite[Section~6]{NO20}
for a description of the maps.
Appealing to Proposition~\ref{prop:Kakeq} we can replace $H_i\left( \mathcal{G}_{A,B} \times_c \ZZ \right)$ with $H_i\left( \mathcal{H}_{A,B} \right)$ and extract the following exact sequence from the one above: 
\begin{equation}\label{eq:exactSubsequence}
\begin{tikzcd}[column sep = 2.1em]
 H_1(\mathcal{H}_{A,B}) \arrow{r}{\color{red} \rho^1} &   H_1(\mathcal{H}_{A,B}) \arrow{r}{\color{Aquamarine} \Phi}  &  H_1(\mathcal{G}_{A,B})    \arrow{r}{\color{blue}\Psi} & H_0(\mathcal{H}_{A,B}) \arrow{r}{\color{Sepia}\rho^0} &   H_0(\mathcal{H}_{A,B}). 
\end{tikzcd}
\end{equation}  
The maps $\Phi$ and $\Psi$ are the unique maps satisfying 
\[ H_1(\pi_\bullet) \circ H_1(\iota)  = \Phi    \quad  \text{and} \quad  \partial_1 = H_0(\iota) \circ \Psi, \]
respectively. Similarly, the maps $\rho^i$ 
 are defined by 
\[ H_i(\iota) \circ \rho^i = \left( \id - H_i(\rho_\bullet) \right) \circ H_i(\iota) \quad  \text{for } i = 0,1.  \]

In the next section we are going to need explicit descriptions of the maps in~\eqref{eq:exactSubsequence}. This is provided in the lemmas below. Some of them are given in terms of ``prefixing an edge'' to path, and therefore we need to assume that the graph~$E_A$ 
has no sources.

\begin{assumption}
For the remainder of this section we assume that the matrix $A$ is essential.
\end{assumption}




\begin{lemma}\label{lem:PhiFormula}
The map $\Phi \colon H_1(\mathcal{H}_{A,B}) \to H_1(\mathcal{G}_{A,B})$ is given by
\[\Phi \left( \left[ 1_{Z(\mu, 1, \mu)} \right] \right) = \left[ 1_{Z(\mu, 1, \mu)} \right] = \left[ 1_{Z(r(\mu), 1, r(\mu))} \right] \in H_1(\mathcal{G}_{A,B}) \] 
 for $\left[ 1_{Z(\mu, 1, \mu)} \right] \in H_1(\mathcal{H}_{A,B})$. 
 In particular, $I(\alpha) = \Phi \left( I_\mathcal{H}(\alpha) \right) \in H_1(\mathcal{G}_{A,B})$ for~${\alpha \in \left\llbracket \mathcal{H}_{A,B} \right\rrbracket}$.
\end{lemma}
\begin{proof}
Straightforward.
\end{proof}

\begin{lemma}\label{lem:rho0Formula} 
The map $\rho^0 \colon H_0(\mathcal{H}_{A,B}) \to H_0(\mathcal{H}_{A,B})$ is given by
\[\rho^0 \left( \left[ 1_{Z(\mu)} \right] \right) = \left[ 1_{Z(\mu)} \right] - \left[ 1_{Z(e \mu)} \right], \]
where $e \in E_A^1$ is any edge with $r(e) = s(\mu)$.
\end{lemma} 
\begin{proof}
We have the following commutative diagram:
\[ \begin{tikzcd}[column sep = large, row sep = large]
H_0(\G_{A,B} \times_c \ZZ) \arrow{r}{\id - H_0(\rho_\bullet)} &   H_0(\G_{A,B} \times_c \ZZ) \\
H_0(\mathcal{H}_{A,B}) \arrow{r}{\rho^0} \arrow{u}{H_0(\iota)}[swap]{\cong} &   H_0(\mathcal{H}_{A,B}) \arrow{u}{H_0(\iota)}[swap]{\cong}
\end{tikzcd} \]
The maps are given by \[H_0(\iota) \left( \left[ 1_{Z(\mu)} \right] \right) = \left[ 1_{Z(\mu) \times \{0\}} \right] \in H_0 \left( \mathcal{G}_{A,B} \times_c \ZZ \right)\]  
and \[H_0\left(\rho_\bullet\right) \left( \left[ 1_{Z(\mu) \times \{0\}} \right] \right) = \left[ 1_{Z(\mu) \times \{1\}} \right] = \left[ 1_{Z(e \mu) \times \{0\}} \right] \in H_0 \left( \mathcal{G}_{A,B} \times_c \ZZ \right), \]
 where $e \in E_A^1$ is any edge with $r(e) = s(\mu)$. Combining these we obtain the desired description of $\rho^0$.
\end{proof}

When $B = 0$, the map $\rho^0 \colon H_0(\mathcal{H}_{A,0}) \to H_0(\mathcal{H}_{A,0})$ coincides with the map \[(\id - \varphi) \colon H_0 \left( \mathcal{H}_{E_A} \right) \to H_0 \left( \mathcal{H}_{E_A} \right),\] where $\varphi$ is from~\cite[Definition~7.5]{NO20}.\footnote{We apologize for the conflicting notation of $\varphi$ with the 1-cocycle from Section~\ref{sec:KEP}, but since the 1-cocycle makes no appearence for the rest of this section we believed it better to stick with the notation from~\cite{NO20} 
to make it easier to compare with results therein.}           
Below, we (trivially) extend the definition of~$\varphi$, as well as~$\varphi^{(k)}$ from~\cite[Definition~8.5]{NO20}, 
 to Katsura--Exel--Pardo groupoids. The automorphism~$\varphi$ is the one induced by $H_0 \left( \rho_\bullet \right)$ when identifying $H_0 \left( \mathcal{G}_{A,B} \times_c \ZZ \right)$ with~$H_0 \left( \mathcal{H}_{A,B} \right)$.


\begin{definition}
Define $\varphi \colon H_0 \left( \mathcal{H}_{A,B} \right) \to H_0 \left( \mathcal{H}_{A,B} \right)$ by for each $\mu \in E_A^*$ setting  
\[\varphi  \left( \left[ 1_{Z(\mu)} \right] \right) = \left[ 1_{Z(e \mu)} \right],  \]
where $e \in E_A^1$ is any edge with $r(e) = s(\mu)$.
For $k \in \ZZ$ we further define
\[\varphi^{(k)} \coloneqq   \begin{cases}
 -(\id + \varphi + \cdots + \varphi^{k-1}) & k > 0, \\
0 & k = 0, \\
 \varphi^{-1} + \varphi^{-2} + \cdots + \varphi^{k} & k < 0.
\end{cases} \]
\end{definition}

\begin{remark}
In the case that $B = 0$ the map $\varphi \colon H_0 \left( \mathcal{H}_{A,0} \right) \to H_0 \left( \mathcal{H}_{A,0} \right)$ coincides the inverse $\delta^{-1}$ of Matui's map $\delta$ from~\cite[page~56]{MatTFG}. See Remarks~7.6 and~8.8 in~\cite{NO20}
for more on this.
\end{remark}

The next lemma is essentially the same as Lemma~8.6 in~\cite{NO20}.

\begin{lemma}\label{lem:PsiFormula}
 Let $[f] \in H_1 \left( \mathcal{G}_{A,B} \right)$ and write $f = \sum_{i=1}^k n_i 1_{Z(\mu_i, 1, \nu_i)}$. Then the map
  $\Psi \colon H_1 \left( \mathcal{G}_{A,B} \right) \to H_0 \left( \mathcal{H}_{A,B} \right)$ is given by%
\[ \Psi([f]) = \sum_{i=1}^k n_i  \varphi^{\left(     \left|  \nu_i \right|   - \left|  \mu_i \right|     \right)} \left(\left[1_{Z(\nu_i)}\right]\right). \]
\end{lemma}
\begin{proof}
Recall that $\partial_1 = H_0(\iota) \circ \Psi$, where $\partial_1 \colon H_1 \left( \mathcal{G}_{A,B} \right) \to H_0 \left( \mathcal{G}_{A,B} \times_c \ZZ \right)$ is the connecting homomorphism in~\eqref{eq:LES}. We are going to describe $\partial_1$ in a similar way as in the proof of~\cite[Lemma~8.6]{NO20}. 
It may be helpful to consult Figure~2 on page~29 of~\cite{NO20}, as we will adopt the notation from there.

Let $[f] \in H_1 \left( \mathcal{G}_{A,B} \right)$ be given, where $f \in C_c\left(\mathcal{G}_{A,B}, \ZZ\right)$ satisfies $\delta_1(f) = 0$. Then we can write $f = \sum_{i=1}^k n_i 1_{Z(\mu_i, 1, \nu_i)}$, where $\sum_{i=1}^N k_i 1_{Z(\mu_i)} = \sum_{i=1}^N k_i 1_{Z(\nu_i)}$. Now view $f + \im(\delta_2)$ as an element in $C_c\left(\mathcal{G}_{A,B}, \ZZ\right) / \im(\delta_2)$.

The element $\pi_1(h) + \im(\delta_2)$, where \[h \coloneqq f \times 0 = \sum_{i=1}^k n_i 1_{Z(\mu_i, 1, \nu_i) \times \{0\}} \in C_c\left(\mathcal{G}_{A,B} \times_c \ZZ, \ZZ\right),\] provides a lift of $f + \im(\delta_2)$ by $\pi_1 + \im(\delta_2)$. Next, we need to compute
 \[\widetilde{\delta_1}(h + \im(\delta_2)) = \delta_1(h) \in C_c\left((\mathcal{G}_{A,B} \times_c \ZZ)^{(0)}, \ZZ\right) \cong C_c\left(E_A^\infty \times \ZZ, \ZZ\right).\] Setting $l_i \coloneqq \left| \mu_i \right| -  \left| \nu_i \right|$ to save space we have
\begin{align*}
\delta_1(h) &= \sum_{i=1}^k n_i (s_* - r_*) \left( 1_{Z(\mu_i, m_i, \nu_i) \times \{0\}} \right) 
\\ &= \sum_{i=1}^k n_i  \left( 1_{s(Z(\mu_i, m_i, \nu_i) \times \{0\})}   - 1_{r(Z(\mu_i, m_i, \nu_i) \times \{0\})}    \right) \\
&=  \sum_{i=1}^k n_i  \left( 1_{Z(\nu_i) \times \{ \left| \mu_i \right| -  \left| \nu_i \right|   \}}   - 1_{Z(\mu_i) \times \{0\})}   \right)
\\ & = \sum_{i=1}^k n_i  \left( 1_{Z(\nu_i) \times \{ l_i  \}}   - 1_{Z(\nu_i) \times \{0\})}   \right),
\end{align*}
where we have used that $\sum_{i=1}^N k_i 1_{Z(\mu_i)} = \sum_{i=1}^N k_i 1_{Z(\nu_i)}$.
By~\cite[Lemma~6.2]{NO20} 
the (unique) lift of $\delta_1(h)$ by $\id - \rho_0$ is the function 
\[g \coloneqq \sum_{i=1}^k n_i L_i, \]
where 
\[ L_i = \begin{cases}
   -  \sum_{j=0}^{l_i -1}   1_{Z(\nu_i) \times \{ j  \}}    \quad      &l_i > 0, \\
0          \quad & l_i = 0, \\
\sum_{j= l_i}^{-1}   1_{Z(\nu_i) \times \{ j  \}}            \quad & l_i < 0.
\end{cases}    \]
Observe that 
\[  \left[ L_i \right] = \varphi^{(l_i)} \left(\left[1_{Z(\nu_i) \times \{ 0  \}}   \right] \right) \in  H_0(\G_{A,B} \times_c \ZZ). \]
This means that 
\[\partial_1([f]) = [g] = \sum_{i=1}^k n_i  \varphi^{\left(\left|  \mu_i \right| - \left|  \nu_i \right| \right)} \left(\left[1_{Z(\nu_i) \times \{0\}}\right]\right) \in H_0(\G_{A,B} \times_c \ZZ),  \]
and hence
\[ \Psi([f]) = \sum_{i=1}^k n_i  \varphi^{\left(     \left|  \nu_i \right|   - \left|  \mu_i \right|     \right)} \left(\left[1_{Z(\nu_i)}\right]\right) \in H_0 \left( \mathcal{H}_{A,B} \right). \]
\end{proof} 

\begin{lemma}\label{lem:PsiPi} 
Assume that $U  \subseteq \mathcal{G}_{A,0} \subseteq \mathcal{G}_{A,B}$ is a full bisection. Let~$I$ and~$I_A$  denote the index maps of $\mathcal{G}_{A,B}$ and $\mathcal{G}_{A,0}$, respectively. If $\Psi \left( I \left( \pi_{U} \right) \right) = 0 \in H_0 \left( \mathcal{H}_{A,B} \right)$, then~${I_A \left( \pi_{U} \right) = 0 \in H_1 \left( \mathcal{G}_{A,0} \right)}$. 
\end{lemma}
\begin{proof}
We can write $U = \sqcup_{i=1}^k Z(\mu_i, 0, \nu_i)$, where $E_A^\infty = \sqcup_{i=1}^k Z(\mu_i) = \sqcup_{i=1}^k Z(\nu_i)$. By Lemma~\ref{lem:PsiFormula}  we have
\[ 0 = \Psi \left( I \left( \pi_{U} \right) \right) = \Psi \left( \left[ 1_U \right] \right) = \sum_{i=1}^k \varphi^{\left(     \left|  \nu_i \right|   - \left|  \mu_i \right|     \right)} \left(\left[1_{Z(\nu_i)}\right]\right) \in \ker \left( \rho^0 \right) \subseteq H_0 \left( \mathcal{H}_{A,B} \right). \]

On the other hand, we have that 
 $H_1 \left( \mathcal{G}_{A,0} \right) \cong \ker \left( \rho^0 \right) \cong \ker \left( \id - H_0 \left( \rho_\bullet \right) \right)$ because ${H_1 \left( \mathcal{G}_{A,0} \times_c \ZZ \right) = 0}$ (see~\cite[Section~7]{NO20}). 
And this isomorphism is implemented by the connecting homomorphism $\partial_1$ from~\eqref{eq:LES} for $B = 0$.  Lemma~8.6 in~\cite{NO20}
(or the proof of Lemma~\ref{lem:PsiFormula} with $B = 0$) says that under this isomorphism the element $I_A \left( \pi_{U} \right) \in H_1 \left( \mathcal{G}_{A,0} \right)$ corresponds to $\Psi \left( I \left( \pi_{U} \right) \right) \in \ker \left( \rho^0 \right)$. Hence $I_A \left( \pi_{U} \right) = 0$.
\end{proof}


The following lemma is part of the proof of~\cite[Proposition~2.5]{Ort}, but we nevertheless sketch the proof for completness.


\begin{lemma}\label{lem:rho1Formula} 
The map $\rho^1 \colon H_1(\mathcal{H}_{A,B}) \to H_1(\mathcal{H}_{A,B})$ is given by
\[\rho^1 \left( \left[ 1_{Z(\mu, m, \mu)} \right] \right) = \left[ 1_{Z(\mu, m, \mu)} \right] - \left[ 1_{Z(e \mu, m, e \mu)} \right], \]
where $e \in E_A^1$ is any edge with $r(e) = s(\mu)$.
\end{lemma}
\begin{proof}
Arguing similarly as in the proof of Lemma~\ref{lem:rho0Formula} it suffices to show that 
\[  \left[ 1_{Z(\mu, 1, \mu) \times \{1\}} \right] = \left[ 1_{Z(e \mu, 1, e \mu) \times \{0\}} \right] \] 
in $H_1\left(\mathcal{G}_{A,B} \times_c \ZZ  \right)$.

Suppose $U$, $V$ are compact bisections with $s(U) = r(V)$ in some ample 
groupoid $\mathcal{G}$. Denote 
\[U \circ V \coloneqq \left( U \times V \right) \cap \mathcal{G}^{(2)} = \left\{  (g,h) \in \mathcal{G}^{(2)} \mid g \in U, \ h \in V \right\}.\]
By~\cite[Lemma~7.3]{MatHom}, we have
\begin{equation}\label{eq:delta2Formula}
\delta_2 \left( 1_{U \circ V}  \right) = 1_U - 1_{U \cdot V} + 1_V.
\end{equation}

Let $e \in E_A^1$ be any edge with $r(e) = s(\mu)$ and define the following bisections in $\mathcal{G}_{A,B} \times_c \ZZ$:
\begin{align*} 
U_1 & \coloneqq Z(\mu, 1,\mu) \times \{1\},    &    V_1 & \coloneqq Z(\mu, 0, e \mu) \times \{1\}, \\ 
U_2 & \coloneqq Z(e \mu, 0,\mu) \times \{0\},	  & 		 V_2 &  \coloneqq Z(\mu, 1, e \mu) \times \{1\}, \\
U_3 & \coloneqq U_2,   &   V_3 &  \coloneqq V_1, \\
U_4 & \coloneqq Z(e \mu, 0,e \mu) \times \{0\}, 	  & 		 V_4  & \coloneqq U_4. 
\end{align*}
From these we define the indicator functions $f_i \coloneqq 1_{U_i \circ V_i} \in C_c \left( \mathcal{G}_{A,B}^{(2)}, \ZZ \right)$ for $i = 1, 2, 3, 4$. Using~\eqref{eq:delta2Formula} it is easy to check that 
\[\delta_2 \left( f_1 + f_2 - f_3 - f_4 \right) = 1_{Z(\mu, 1, \mu) \times \{1\}}   -    1_{Z(e \mu, 1, e \mu) \times \{0\}}, \] 
which shows that $\left[ 1_{Z(\mu, 1, \mu) \times \{1\}} \right] = \left[ 1_{Z(e \mu, 1, e \mu) \times \{0\}} \right]$ in $H_1\left(\mathcal{G}_{A,B} \times_c \ZZ  \right)$.
\end{proof}

\begin{remark} 
The main result of~\cite{Ort} is the following description of the homology groups of $\mathcal{G}_{A,B}$, assuming that the self-similar graph is pseudo-free: 
\begin{align*}
H_0 \left( \mathcal{G}_{A,B} \right) & \cong \coker \left(I_N - A  \right), \\
H_1 \left( \mathcal{G}_{A,B} \right) & \cong \ker \left(I_N - A  \right) \oplus \coker \left(I_N - B  \right), \\
H_2 \left( \mathcal{G}_{A,B} \right) & \cong \ker \left(I_N - B  \right), \\
H_i \left( \mathcal{G}_{A,B} \right) &= 0, \quad i \geq 3.
\end{align*} 
Here $I_N$ is the $N \times N$ identity matrix and $I_N - A$, $I_N - B$ are viewed as endomorphisms of~$\mathbb{Z}^N$. 
When the self-similar graph is pseudo-free, \cite[Lemma~2.2]{Ort} shows that $H_i \left( \mathcal{H}_{A,B} \right) = 0$ for $i \geq 2$. This truncates the long exact sequence~\eqref{eq:LES} into (identifying as in~\eqref{eq:exactSubsequence}): 
\[ \begin{tikzcd}[column sep = 2.7em]
0 \arrow{r} &[-20pt]    H_2(\mathcal{G}_{A,B}) \arrow{r} &[-7pt] H_1(\G_{A,B} \times_c \ZZ) \arrow{r}{\color{red}  \rho^1} &   H_1(\G_{A,B} \times_c \ZZ) \arrow{r}{\color{Aquamarine} \Phi} & \textcolor{Aquamarine}{(a)}  \\ 
\textcolor{Aquamarine}{(a)} \arrow{r} &  H_1(\mathcal{G}_{A,B})    \arrow{r}{\color{blue} \Psi} & H_0(\G_{A,B} \times_c \ZZ) \arrow{r}{\color{Sepia} \rho^0} &   H_0(\G_{A,B} \times_c \ZZ) \arrow{r} & H_0(\mathcal{G}_{A,B})  \arrow{r} &[-20pt] 0.
\end{tikzcd}  \]
It follows that $H_2 \left( \mathcal{G}_{A,B} \right) \cong \ker \left( \rho^1 \right)$, $H_0 \left( \mathcal{G}_{A,B} \right) \cong \coker \left( \rho^0 \right)$ and that
\begin{equation}\label{eq:SES}
\begin{tikzcd}
0 \arrow{r} & \coker \left( \rho^1 \right) \arrow{r}{\widetilde{\Phi}} &  H_1(\mathcal{G}_{A,B})    \arrow{r}{ \Psi} & \ker \left( \rho^0 \right)  \arrow{r} & 0
\end{tikzcd}
\end{equation}
is exact. 
It is also showed in~\cite{Ort} that 
\begin{align*}
\ker \left( \rho^0 \right) & \cong \ker \left(I_N - A  \right),		&		\coker \left( \rho^0 \right) & \cong \coker \left(I_N - A  \right), \\ 
\ker \left( \rho^1 \right) & \cong \ker \left(I_N - B  \right),	&		\coker \left( \rho^1 \right)& \cong \coker \left(I_N - B  \right).
\end{align*}
Since $\ker \left( \rho^0 \right)$ is free, the exact 
sequence~\eqref{eq:SES} splits, and we therefore obtain an isomorphism 
${H_1(\mathcal{G}_{A,B})   \cong  \ker \left( \rho^0 \right)     \oplus  \coker \left( \rho^1 \right)}$. 

We remark that these results are valid for $N = \infty$ as well. Moreover, the descriptions of~$H_0 \left( \mathcal{G}_{A,B} \right)$ and $H_1 \left( \mathcal{G}_{A,B} \right)$ are valid even when the self-similar graph is not pseudo-free.
\end{remark}

\section{Property~TR for $\mathcal{G}_{A,B}$}\label{sec:TR}

The aim of this section is to show that the Katsura--Exel--Pardo groupoid $\mathcal{G}_{A,B}$ has Property~TR. This means that given ${\alpha \in \left\llbracket \mathcal{G}_{A,B} \right\rrbracket}$ with $I(\alpha) = 0$, we need to show that~${\alpha \in \mathcal{T}(\mathcal{G}_{A,B})}$. In a nutshell, the strategy is to decompose the topological full group as
 \[\left\llbracket \mathcal{G}_{A,B} \right\rrbracket = \left\llbracket \mathcal{H}_{A,B} \right\rrbracket  \left\llbracket \mathcal{G}_{A,0} \right\rrbracket \] and show that Property~TR is inherited from the kernel groupoid $\mathcal{H}_{A,B}$ and the SFT-groupoid~$\mathcal{G}_{A,0}$. In what follows we will view~${\mathcal{G}_{A,0} \cong \mathcal{G}_{A}}$ as a subgroupoid of $\mathcal{G}_{A,B}$.

\begin{assumption}
In this whole section we fix $N \times N$ matrices $A,B$ which satisfy the AH~criteria and where $A$ is essential. In particular $N < \infty$ and $A$ is an irreducible non-permutation matrix.
\end{assumption}


\begin{proposition}\label{prop:IHonto}
The index map $I_\mathcal{H} \colon  \left\llbracket \mathcal{H}_{A,B} \right\rrbracket \to H_1\left(\mathcal{H}_{A,B}\right)$ is surjective.
\end{proposition}
\begin{proof}
Let $\mu \in E_A^*$ and consider the bisection $V \coloneqq Z(\mu,1,\mu) \subseteq \mathcal{H}_{A,B}$. Since we have ${s\left(V \right) = r\left(V \right) = Z(\mu)}$ we can define a full bisection $U \coloneqq V \sqcup \left( E_A^\infty \setminus Z(\mu) \right) \subseteq \mathcal{H}_{A,B}$. By~\cite[Lemma~7.3]{MatHom} we have
$I_\mathcal{H} \left( \pi_U \right) = \left[ 1_V  \right]$. The result now follows since these elements span $H_1\left(\mathcal{H}_{A,B}\right)$ (by~\eqref{eq:limitH} and~\eqref{eq:H1span}).
\end{proof}

\begin{lemma}\label{lem:HnTR}
For each $n \in \mathbb{N}$ the groupoid $\mathcal{H}_{A,B,n}$ has Property~TR.
\end{lemma}
\begin{proof}
Let $U \subseteq \mathcal{H}_{A,B,n}$ be a full bisection. Then  $U = \sqcup_{i=1}^k Z(\mu_i, m_i, \nu_i)$, where $\mu_i, \nu_i \in E_A^{\leq n}$ satisfy $\left| \mu_i \right| = \left| \nu_i \right|$, $r(\mu_i) = r(\nu_i)$ and $E_A^\infty = \sqcup_{i=1}^k Z(\mu_i) = \sqcup_{i=1}^k Z(\nu_i)$. Using the fact that each basic bisection decomposes as
\begin{equation}\label{eq:decompose}
Z(\mu, m, \nu) = \bigsqcup_{e \in s^{-1}(r(\nu))} Z(\mu \kappa_m(e) , \varphi(m, e), \nu e)
\end{equation}
we can assume without loss of generality that $\left| \mu_i \right| = \left| \nu_i \right| = n$ for all $1 \leq i \leq k$. We may also set $m_i = 0$ whenever $r(\mu_i)$ is not $B$-regular, for then $Z(\mu_i, m_i, \nu_i) = Z(\mu_i, 0, \nu_i)$, by the same reasoning as in Equation~\eqref{eq:degenerateBisection}.

Let us now consider the index map $I_{\mathcal{H}, n} \colon  \left\llbracket \mathcal{H}_{A,B,n} \right\rrbracket \to H_1(\mathcal{H}_{A,B,n})$. Using~\eqref{eq:indicatorHomology} we compute 
\[ I_{\mathcal{H}, n} \left( \pi_U \right) = \left[ 1_U \right] = \sum_{i=1}^k \left[1_{Z(\mu_i, m_i, \nu_i)} \right] = \sum_{i=1}^k m_i \left[1_{Z(\mu_i, 1, \mu_i)} \right]  \in  H_1(\mathcal{H}_{A,B,n}). \]
For each vertex $v \in E_A^0$ let $\mathcal{I}_v \coloneqq \left\{1 \leq i \leq k \mid r(\mu_i) = v  \right\}$. 
Using the identification in~\eqref{eq:Homn} (where only the $B$-regular vertices matter) we see that $I_{\mathcal{H}, n} \left( \pi_U \right) = 0$ if and only if ${\sum_{i \in \mathcal{I}_v} m_i = 0}$ for each vertex $v \in E_A^0$. 

We define two more full bisections in $\mathcal{H}_{A,B,n}$, namely 
\[ U_\mathcal{H} \coloneqq \sqcup_{i=1}^k Z(\mu_i, m_i, \mu_i)   \quad \text{ and }		 \quad		 U_A \coloneqq \sqcup_{i=1}^k Z(\mu_i, 0, \nu_i).   \]
Observe that $U_\mathcal{H}  \cdot U_A = U$. We claim that $\pi_{U_A}$ is a product of transpositions, i.e. that~${\pi_{U_A} \in \mathcal{T}(\mathcal{H}_{A,B,n})}$. Indeed, since $E_A^\infty = \sqcup_{i=1}^k Z(\mu_i) = \sqcup_{i=1}^k Z(\nu_i)$ and $\left| \mu_i \right| = \left| \nu_i \right| = n$, we must have that $E_A^n = \{\mu_1, \mu_2, \ldots, \mu_k \} = \{\nu_1, \nu_2, \ldots, \nu_k \}$. Hence the homeomorphism $\pi_{U_A}$ on $E_A^\infty$ can be identified with a permutation on a finite set of $k$ symbols which maps~$\nu_i$ to~$\mu_i$. The claim then follows.

Next we will show that $\pi_{U_\mathcal{H}}$ is a product of transpositions when $I_{\mathcal{H}, n} \left( \pi_U \right) = 0$. Let~$\mathcal{I}^0$ denote the set of vertices $v$ for which $\mathcal{I}_v \neq \emptyset$ and pick a distinguished index $i_v \in \mathcal{I}_v$ for each vertex $v \in \mathcal{I}^0$.
Suppose that $r(\mu_{i_1}) = v = r(\mu_{i_v})$ for some index $i_1 \neq i_v$. 
Set~$V_1 \coloneqq Z(\mu_{i_v}, m_{i_1}, \mu_{i_1})$ and $W_1 \coloneqq Z(\mu_{i_v}, 0, \mu_{i_1})$. Then
\[U_\mathcal{H} \cdot \widehat{V_1} \cdot \widehat{W_1} = \left( \bigsqcup_{i \neq i_v, i_1} Z(\mu_i, m_i, \mu_i) \right) \bigsqcup Z(\mu_{i_v},m_{i_v} + m_{i_1}, \mu_{i_v}) \bigsqcup Z(\mu_{i_1}, 0, \mu_{i_1}). \]
By iterating this process enough times for each vertex we can write 
\begin{equation}\label{eq:iterate}
U_\mathcal{H} \cdot \widehat{V_1} \cdot \widehat{W_1} \cdots \widehat{V_K} \cdot \widehat{W_K} =  
\bigsqcup_{v \in \mathcal{I}^0} \left(  Z\left( \mu_{i_v},{\textstyle \sum_{i \in \mathcal{I}_v} m_{i}}, \mu_{i_v} \right)              
\sqcup    \bigsqcup_{i \in \mathcal{I}_v \setminus \{i_v\}}  Z(\mu_i, 0, \mu_i) \right),
\end{equation}
where the $V_i, W_i$'s are compact bisections with disjoint source and range, so that $\pi_{\widehat{V_i}}, \pi_{\widehat{W_i}}$ are transpositions.
Now if $I_{\mathcal{H}, n} \left( \pi_U \right) = 0$, then each $\sum_{i \in \mathcal{I}_v} m_{i} = 0$, in which case~\eqref{eq:iterate} says that 
\[  \pi_{U_\mathcal{H}} \left( \pi_{\widehat{V_1}} \pi_{\widehat{W_1}} \cdots \pi_{\widehat{V_K}} \pi_{\widehat{W_K}} \right) = \id_{E_A^\infty}.  \]
This shows that $\pi_{U_\mathcal{H}}  \in \mathcal{T}(\mathcal{H}_{A,B,n})$ and hence $\pi_U = \pi_{U_\mathcal{H}} \pi_{U_A} \in  \mathcal{T}(\mathcal{H}_{A,B,n})$ too. 
\end{proof}

\begin{proposition}\label{prop:HTR}
The groupoid $\mathcal{H}_{A,B}$ has Property~TR.
\end{proposition}
\begin{proof}
Since $\mathcal{H}_{A,B} = \cup_{n=1}^\infty \mathcal{H}_{A,B,n}$ and ${\mathcal{H}_{A,B,n}}^{(0)} = \mathcal{H}_{A,B}^{(0)}$ is compact, we also have that ${\left\llbracket \mathcal{H}_{A,B} \right\rrbracket = \cup_{n=1}^\infty \left\llbracket \mathcal{H}_{A,B,n} \right\rrbracket}$. Suppose that $I_\mathcal{H}(\pi_U) = 0 \in H_1(\mathcal{H}_{A,B})$ for some $\pi_U \in \left\llbracket \mathcal{H}_{A,B} \right\rrbracket$.
We have $\pi_U \in \left\llbracket \mathcal{H}_{A,B,n} \right\rrbracket$ for some $n$. By~\eqref{eq:limitH} we must have $I_{\mathcal{H}, n'} \left( \pi_U \right) = 0$ for some $n' \geq n$. The result now follows from Lemma~\ref{lem:HnTR}.
\end{proof}

\begin{remark}
Even though $\mathcal{H}_{A,B}$ is minimal and has Property~TR, Proposition~4.5 in~\cite{MatProd} does not apply, because $\mathcal{H}_{A,B}$ is not purely infinite and generally not principal. 
\end{remark}

Recall the exact sequence~\eqref{eq:exactSubsequence} from the previous section, as it is going to be used in the proofs of the next two results. The following lemma is inspired by~\cite[Lemma~4.7]{MatProd}.


\begin{lemma}\label{lemma:HtoG}
Let $U \subseteq \mathcal{H}_{A,B}$ be a full bisection and view $\pi_U$ as an element of $\left\llbracket \mathcal{G}_{A,B} \right\rrbracket$. If $I(\pi_U) = 0 \in H_1\left(\mathcal{G}_{A,B}\right)$, then $\pi_U \in \mathcal{T}\left(\mathcal{G}_{A,B}\right)$.
\end{lemma}
\begin{proof}
Set $\alpha \coloneqq \pi_U$. By Lemma~\ref{lem:PhiFormula} we have $\Phi \left( I_\mathcal{H}(\alpha) \right) = I(\alpha)  = 0$, so  $I_\mathcal{H}(\alpha) \in \ker(\Phi) = \im(\rho^1)$. Let $[f] \in H_1\left(\mathcal{H}_{A,B}\right)$ be such that $I_\mathcal{H}(\alpha) = \rho^1([f])$. By Proposition~\ref{prop:IHonto} there is some $\beta \in \left\llbracket \mathcal{H}_{A,B} \right\rrbracket$ such that $I_\mathcal{H}(\beta) = [f]$. 

Now $\beta = \pi_V$ for some full bisection $V = \sqcup_{i=1}^k Z(\mu_i, m_i, \nu_i) \subseteq \mathcal{H}_{A,B}$, where we have ${E_A^\infty = \sqcup_{i=1}^k Z(\mu_i) = \sqcup_{i=1}^k Z(\nu_i)}$ and $\left| \mu_i \right| = \left| \nu_i \right| = n$ for all $i$, for some $n \in \mathbb{N}$. Employing the same argument and notation as in the proof of Lemma~\ref{lem:HnTR} we can find a product of transpositions $\beta_0 \in \mathcal{T} \left( \mathcal{H}_{A,B} \right)$ such that $\beta \beta_0 = \pi_W$, where $W =   \left(\sqcup_{v \in \mathcal{I}^0}  Z\left( \mu_{i_v}, l_{i_v}, \mu_{i_v} \right) \right) \sqcup A$ with $A \subseteq \mathcal{H}_{A,B}^{(0)}$ and $ l_{i_v} \in \mathbb{Z}$. In particular, the paths $\mu_{i_v}$ all have different ranges.

For each $v \in \mathcal{I}^0$ pick an edge $e_v \in E_A^1$ with $r(e_v) = s \left( \mu_{i_v} \right) \neq s(e_v)$, so that $e_v$ is not a loop. 
Then for each $v$, the path $e_v \mu_{i_v}$ is disjoint from $\mu_{i_v}$. Since all the $\mu_{i_v}$'s are mutually disjoint, so are all the $e_v \mu_{i_v}$'s too. A priori, it is not guaranteed that $\mu_{i_v}$ is disjoint from $e_w \mu_{i_w}$ when~${v \neq w \in \mathcal{I}^0}$. However, this (i.e.\ that $\mu_{i_v} \nleq e_w \mu_{i_w}$) can be arranged if we at the start ensure that $n$ is chosen large enough (which in turn can be done by~\eqref{eq:decompose}) so that $\left| E_A^{n-1} v \right| \geq 2N$ for each $v \in E_A^0$. For this gives enough options when choosing the distinguished indices $i_v$ to ensure that the total collection of paths $\cup_{v \in \mathcal{I}^0} \left\{ \mu_{i_v}, e_v \mu_{i_v} \right\}$ are mutually disjoint (independent of the choice of the $e_v$'s).   

By the above paragraph we may define the compact bisection \[T \coloneqq \sqcup_{v \in \mathcal{I}^0} Z \left(e_v \mu_{i_v} , 0, \mu_{i_v} \right) \subseteq \mathcal{G}_{A,B},\] which has disjoint source and range. Define $\tau_T \coloneqq \pi_{\widehat{T}} \in \mathcal{T}(\mathcal{G}_{A,B})$. Observe that we have $\widehat{T} \cdot W \cdot \widehat{T} = \left(\sqcup_{v \in \mathcal{I}^0}  Z\left(e_v \mu_{i_v}, l_{i_v}, e_v \mu_{i_v} \right) \right) \sqcup A'$ with $A' \subseteq \mathcal{H}_{A,B}^{(0)}$. Combining this with the description of $\rho^1$ from Lemma~\ref{lem:rho1Formula}  we see that
\begin{align}\label{eq:rho1W}
\rho^1 \left(  I_\mathcal{H} \left( \pi_W \right) \right) &= \rho^1 \left( \left[ 1_W \right] \right) =  \left[ 1_W \right] - \left[ 1_{\widehat{T} \cdot W \cdot \widehat{T}} \right] \nonumber \\ &=  I_\mathcal{H} \left( \pi_W \right) - I_\mathcal{H} \left(\tau_T \pi_W \tau_T \right) = I_\mathcal{H} \left(\pi_W  \tau_T \pi_W^{-1} \tau_T \right). 
\end{align}
At the same time we have
\begin{equation}\label{eq:Wbeta}
I_\mathcal{H} \left( \pi_W \right)  = I_\mathcal{H} \left( \beta \right) = [f],
\end{equation}
since $\pi_W = \beta \beta_0$ and $\beta_0 \in \mathcal{T} \left( \mathcal{H}_{A,B} \right)$. Next we observe that 
\[ W \cdot \widehat{T} \cdot W^{-1} = \bigsqcup_{v \in \mathcal{I}^0} \left(  Z\left(e_v \mu_{i_v}, - l_{i_v},  \mu_{i_v} \right) \sqcup Z\left( \mu_{i_v}, l_{i_v}, e_v \mu_{i_v} \right) \right) \sqcup A'',    \]
where $A'' \subseteq \mathcal{G}_{A,B}^{(0)}$. This actually shows that $\pi_W  \tau_T \pi_W^{-1} \in \mathcal{T} \left( \mathcal{G}_{A,B} \right)$, since  $W \cdot \widehat{T} \cdot W^{-1}  = \pi_{\widehat{R}}$, where $R = \sqcup_{v \in \mathcal{I}^0}   Z\left( \mu_{i_v}, l_{i_v}, e_v \mu_{i_v} \right)$. Define the element $\gamma \coloneqq \pi_W  \tau_T \pi_W^{-1} \tau_T \in\mathcal{T} \left( \mathcal{G}_{A,B} \right)$. Equations~\eqref{eq:rho1W} and~\eqref{eq:Wbeta} now says that 
\[ I_\mathcal{H} \left( \gamma \right)  =  \rho^1 \left(  I_\mathcal{H} \left( \pi_W \right) \right) = \rho^1 \left( \left[ f \right] \right) =   I_\mathcal{H} \left( \alpha \right).   \]
This means that $ I_\mathcal{H} \left( \alpha \gamma^{-1} \right) = 0 \in H_1 \left( \mathcal{H}_{A,B} \right)$, and hence $\alpha \gamma^{-1} \in \mathcal{T} \left( \mathcal{H}_{A,B} \right)$ by Proposition~\ref{prop:HTR}. But then $\alpha \in \mathcal{T} \left( \mathcal{G}_{A,B} \right)$ and we are done.
\end{proof}

%


\begin{theorem}\label{thm:TRKEP}
The Katsura--Exel--Pardo groupoid $\mathcal{G}_{A,B}$ has Property~TR.
\end{theorem}
\begin{proof}
Let $U \subseteq \mathcal{G}_{A,B}$ be a full bisection. Then we can write $U = \sqcup_{i=1}^k Z(\mu_i, m_i, \nu_i)$, where~${E_A^\infty = \sqcup_{i=1}^k Z(\mu_i) = \sqcup_{i=1}^k Z(\nu_i)}$ (but the paths $\mu_i$ and $\nu_i$ may now have different lengths).
As in the proof of Lemma~\ref{lem:HnTR} we define the full bisections 
\[ U_\mathcal{H} \coloneqq \sqcup_{i=1}^k Z(\mu_i, m_i, \mu_i) \subseteq \mathcal{H}_{A,B}     		 \quad \text{ and } \quad 		  U_A \coloneqq \sqcup_{i=1}^k Z(\mu_i, 0, \nu_i) \subseteq \mathcal{G}_{A,0},   \]
where we view both $\mathcal{H}_{A,B}$ and $\mathcal{G}_{A,0}$ as subgroupoids of $\mathcal{G}_{A,B}$. Recall that we have $U_\mathcal{H}  \cdot U_A = U$ and $\pi_U = \pi_{U_\mathcal{H}} \pi_{U_A} \in \left\llbracket \mathcal{G}_{A,B} \right\rrbracket$. We will be considering all three index maps: 
\begin{align*}
 I \colon & \left\llbracket \mathcal{G}_{A,B} \right\rrbracket \to H_1(\mathcal{G}_{A,B}), \\
 I_\mathcal{H} \colon & \left\llbracket \mathcal{H}_{A,B} \right\rrbracket \to H_1(\mathcal{H}_{A,B}), \\ 
 I_A \colon & \left\llbracket \mathcal{G}_{A,0} \right\rrbracket \to  H_1(\mathcal{G}_{A,0}). 
\end{align*}
We have that $I \left( \pi_U \right) =  I \left( \pi_{U_\mathcal{H}} \right) + I \left( \pi_{U_A} \right) \in H_1 \left( \mathcal{G}_{A,B} \right)$, but by viewing ${\pi_{U_\mathcal{H}} \in \left\llbracket \mathcal{H}_{A,B} \right\rrbracket}$ and~${\pi_{U_A} \in \left\llbracket \mathcal{G}_{A,0} \right\rrbracket}$ we may also consider  
$I_\mathcal{H} \left( \pi_{U_\mathcal{H}} \right)$ and $I_A \left( \pi_{U_A} \right)$ as elements of~$H_1 \left( \mathcal{H}_{A,B} \right)$ and~$H_1 \left( \mathcal{G}_{A,0} \right)$, respectively.
The idea is to show that if~${I \left( \pi_U \right) = 0}$, then both $I \left( \pi_{U_\mathcal{H}} \right)$ and $I_A \left( \pi_{U_A} \right)$ vanish as well.
For then we may appeal to Lemma~\ref{lemma:HtoG} and Property~TR for~${\mathcal{G}_{A,0} \cong \mathcal{G}_A}$, respectively, to conclude that~$\pi_U$ itself must be a product of transpositions.

Assume now that $I \left( \pi_U \right) = 0 \in H_1 \left( \mathcal{G}_{A,B} \right)$.   
By Lemma~\ref{lem:PhiFormula} and the exactness of~\eqref{eq:exactSubsequence} have
\[\Psi \left( I \left( \pi_{U_\mathcal{H}} \right) \right) = \Psi \left(\Phi \left(I_\mathcal{H} \left( \pi_{U_\mathcal{H}} \right)  \right) \right) = 0.\]
This means that 
$\Psi \left( I \left( \pi_{U_A} \right) \right) = \Psi \left( I \left( \pi_U \right) \right) = 0$. From Lemma~\ref{lem:PsiPi} we conclude that ${I_A \left( \pi_{U_A} \right) = 0 \in H_1 \left( \mathcal{G}_{A,0} \right)}$. Hence $\pi_{U_A} \in \mathcal{T} \left( \mathcal{G}_{A,0} \right) \subseteq \mathcal{T} \left( \mathcal{G}_{A,B} \right)$ by appealing to Property~TR for SFT-groupoids~\cite{MatTFG}. It follows  that $I \left( \pi_{U_A} \right) = 0 \in H_1 \left( \mathcal{G}_{A,B} \right)$ too, and then ${I \left( \pi_{U_\mathcal{H}} \right) = I \left( \pi_U \right)  = 0 \in H_1 \left( \mathcal{G}_{A,B} \right)}$. By Lemma~\ref{lemma:HtoG} we then get $\pi_{U_\mathcal{H}} \in \mathcal{T} \left( \mathcal{G}_{A,B} \right)$ as well. This finishes the proof, since $\pi_U = \pi_{U_\mathcal{H}} \pi_{U_A}$.
\end{proof}

\begin{corollary}\label{cor:AH}
The AH~conjecture holds for the Katsura--Exel--Pardo groupoid $\mathcal{G}_{A,B}$ whenever the matrices $A,B$ satisfy the AH~criteria and $A$ is irreducible. 
\end{corollary}
\begin{proof}
Since $\mathcal{G}_{A,B}$ has Property~TR (Theorem~\ref{thm:TRKEP}) and is purely infinite (Proposition~\ref{prop:PI}) and minimal, the result follows from~\cite[Theorem~4.4]{MatProd}.
\end{proof}

\begin{remark}
To get rid of the assumption of $A$ being essential, i.e.\ allowing for sources in~$E_A$, we need to prove Property~TR for restrictions, as is done for graph groupoids in~\cite{NO20}.
This should be doable. 
\end{remark}

\section{Finite generation of $\left\llbracket \mathcal{G}_{A,B} \right\rrbracket$}\label{sec:finGen}

In this section we will show that the topological full group $\left\llbracket \mathcal{G}_{A,B} \right\rrbracket$ is finitely generated, under the following hypotheses on $A$ and $B$. 

\begin{assumption}
In this whole section we fix $N \times N$ matrices $A,B$ which satisfy the AH~criteria and where $A$ is essential. In particular $N < \infty$ and $A$ is an irreducible non-permutation matrix. Furthermore, we assume that $\left| B_{i,j} \right| < A_{i,j}$ whenever $A_{i,j} \neq 0$. 
\end{assumption}

In~\cite[Definition~5.1]{Nek18}, Nekrashevych defined the notion of a self-similar group being \emph{contracting}. He showed that for a contracting self-similar group, the topological full group of the associated groupoid of germs is finitely presented. We will extend Nekrashevych's definition to cover the self-similar graphs of Exel and Pardo, and show that the self similar graph associated to matrices $A$ and $B$ as above is contracting. However, we will settle for showing that $\left\llbracket \mathcal{G}_{A,B} \right\rrbracket$ is finitely generated. A crucial ingredient in our argument is the fact that the topological full group $\left\llbracket \mathcal{G}_{A,0} \right\rrbracket$ is finitely generated~\cite{MatTFG}.


\begin{definition}\label{def:contracting}
A self-similar graph $(G, E, \varphi)$ is \emph{contracting} if there exists a finite subset~${\mathcal{N} \subset G}$ with the property that for every $g \in G$ there is some $n \in \mathbb{N}$ such that $\varphi(g, \mu) \in \mathcal{N}$ for all $\mu \in E^{\geq n}$.
\end{definition}

The following rudimentary lemma will be used when showing that the self-similar graph from Section~\ref{sec:KEP} is contracting.

\begin{lemma}\label{lem:inequalities}
Assume that $a,b,m,t \in \ZZ$ are integers satisfying $a \geq 1$ and
\begin{align*}
(b-a) m - a &< a t < (b-a) m + a, \\
1 - 2a & \leq b-a \leq -1.
\end{align*}
Then 
\begin{align*}
\left| m+t \right| & < \left| m \right| \quad  \text{when } \left| m \right| \geq 2a, \\
\left| m+t \right| & \leq  \left| m \right| \quad  \text{when } \left| m \right| < 2a.
\end{align*}
\end{lemma}
\begin{proof}
Assume first that $m \geq 0$. Then 
\[ (b-a) m - a \geq (1-2a) m - a = m - a - 2ma  \]
and
\[(b-a) m + a \leq (-1) m + a = a-m,\]
so 
\[m - a - 2ma < a t < a-m.  \]
Now if $m \geq 2a$, then 
\[ a-m \leq -a \quad \text{and} \quad m - a - 2ma \geq a -2ma = (1-2m)a.  \]
We infer from this that 
\[1-2m < t < -1 \implies -2m < t < 0 \implies \left| m + t  \right| < \left| m \right|.  \]

Next, if $0 \leq m < a$, then 
\[ a-m \leq a \quad \text{and} \quad m - a - 2ma \geq -a -2ma = (-1-2m)a.  \]
And from this we get 
\[-1-2m < t < 1 \implies  -2m \leq t \leq 0 \implies \left| m + t  \right| \leq \left| m \right|.  \]

The case $m < 0$ proceeds in essentially the same way.
\end{proof}

\begin{lemma}\label{lem:varphi}
Let $e = e_{i,j,n} \in E_A^1$ and $m \in \mathbb{Z}$ be given. Then we have
\begin{align*}
\left| \varphi(m,e) \right| & < \left| m \right| \quad  \text{when } \left| m \right| \geq 2 A_{i,j}, \\
\left| \varphi(m,e) \right| & \leq  \left| m \right| \quad  \text{when } \left| m \right| < 2 A_{i,j}.
\end{align*}
\end{lemma}
\begin{proof}
We have that $A_{i,j} \geq 1$ and that $0 \leq n < A_{i,j}$. Let $t$ and $r$ be the unique integers satisfying 
\[m B_{i,j} + n = (m+t) A_{i,j} + r \quad  \text{and} \quad  0 \leq r < A_{i,j}.  \]
Recall that then $\varphi(m,e) = m+t$. We now have 
\[ \left(B_{i,j} - A_{i,j}  \right) m  + n - r = A_{i,j} t  \] 
where $ - A_{i,j} < n-r <  A_{i,j}$. From this we see that 
\[ \left(B_{i,j} - A_{i,j}  \right) m - A_{i,j}  <  A_{i,j} t  < \left(B_{i,j} - A_{i,j}  \right) m + A_{i,j}. \]
We also have 
\[  1 - 2 A_{i,j} \leq B_{i,j} - A_{i,j} \leq -1, \]
since $\left| B_{i,j} \right| < A_{i,j}$. We are now in the setting of Lemma~\ref{lem:inequalities} and so the result follows. 
\end{proof}

\begin{proposition}\label{prop:contracting}
The self-similar graph $(\ZZ, E_A, \varphi)$ associated to the matrices $A$ and $B$ is contracting. 
\end{proposition}
\begin{proof}
Define $R \coloneqq 2 \cdot \max \{ A_{i,j} \mid 1 \leq i,j \leq N \}$. Let $m \in \ZZ$ be given. Combining Lemma~\ref{lem:varphi} with Equation~\eqref{eq:iteratephi} we see that $\left| \varphi(m,\mu) \right| \leq R$ whenever $\mu \in E_A^{\geq \left| m \right|}$. So by setting $\mathcal{N} = \left[ -R, R \right] \cap \ZZ$ and $n = \left| m \right|$ in Definition~\ref{def:contracting} we find that $(\ZZ, E_A, \varphi)$ is contracting. 
\end{proof}

Before establishing the finite generation of $\left\llbracket \mathcal{G}_{A,B} \right\rrbracket$ we introduce some notation. Given a finite path $\gamma \in E_A^*$ and $m \in \mathbb{Z}$ we denote the full bisection $Z \left( \gamma, m, \gamma \right) \sqcup \left( \mathcal{G}_{A,B}^{(0)} \setminus Z(\gamma) \right)$ by $U_{\gamma, m}$. Given two disjoint paths $\mu, \gamma \in E_A^*$ with $r(\mu) = r(\gamma)$ we define the transposition~${\tau_{\mu, \gamma} \coloneqq \pi_{\widehat{V}} \in \left\llbracket \mathcal{G}_{A,0} \right\rrbracket}$, where $V = Z \left( \mu, 0, \gamma \right)$. Observe that we have 
\[\tau_{\mu, \gamma} \circ \pi_{U_{\gamma, m}} \circ \tau_{\mu, \gamma}  = \pi_{U_{\mu, m}}. \]

\begin{theorem}\label{thm:finGen}
Let $A,B$ be matrices satisfying the AH~criteria with $A$ irreducible. Assume that $\left| B_{i,j} \right| < A_{i,j}$ whenever $A_{i,j} \neq 0$. Then the topological full group $\left\llbracket \mathcal{G}_{A,B} \right\rrbracket$ is finitely generated.
\end{theorem}
\begin{proof}
First pick $n \in \mathbb{N}$ large enough so that $E_A^n v \geq 2$ for each $v \in E_A^0$. As above, set~${R \coloneqq 2 \cdot \max \{ A_{i,j} \mid 1 \leq i,j \leq N \}}$.  Let $S$ be a finite generating set for $\left\llbracket \mathcal{G}_{A,0} \right\rrbracket$ (\cite[Theorem~6.21]{MatTFG}) and define the finite set
\[T \coloneqq \left\{ \pi_{U_{\gamma, m}} \mid \gamma \in E_A^n \ \& \ -R \leq m \leq R \right\}.\]
We claim that $S \cup T$ generates $\left\llbracket \mathcal{G}_{A,B} \right\rrbracket$.

To prove the claim let $\pi_U \in \mathcal{G}_{A,B}$ be given.  Write $U = \sqcup_{i=1}^k Z(\mu_i, m_i, \nu_i)$. 
By applying the splitting in Equation~\eqref{eq:decompose} enough times to each basic bisection in~$U$ we may  assume without loss of generality that $\left| \mu_i \right| \geq n$ for each $i$. Similarly, by Proposition~\ref{prop:contracting} we may assume that $\left| m_i \right| \leq R$. As we have done a few times already we split the full bisection $U$ into the two full bisections
\[ U_\mathcal{H} \coloneqq \sqcup_{i=1}^k Z(\mu_i, m_i, \mu_i)    \quad \text{ and } \quad U_A \coloneqq \sqcup_{i=1}^k Z(\mu_i, 0, \nu_i),   \]
making $\pi_U = \pi_{U_\mathcal{H}} \pi_{U_A}$. Since $\pi_{U_A} \in \left\llbracket \mathcal{G}_{A,0} \right\rrbracket$ and $\pi_{U_\mathcal{H}} = \Pi_{i=1}^k \pi_{U_{\mu_i, m_i}}$ it suffices to consider each $\pi_{U_{\mu_i, m_i}}$. By the assumption on $n$ we can for each $i$ find a path $\gamma_i \in E_A^n r(\mu_i)$ which is disjoint from $\mu_i$. And then the equation
\[\pi_{U_{\mu_i, m_i}} = \tau_{\mu_i, \gamma_i} \circ \pi_{U_{\gamma_i, m_i}} \circ \tau_{\mu_i, \gamma_i} \]
proves the claim, since $\tau_{\mu_i, \gamma_i} \in \left\llbracket \mathcal{G}_{A,0} \right\rrbracket$ and $\pi_{U_{\gamma_i, m_i}} \in T$.
\end{proof}



\newcommand{\etalchar}[1]{$^{#1}$}

\end{document}